\documentclass[11pt,leqno]{amsart}

\usepackage{graphicx}
\usepackage{geometry}
\usepackage{amsmath}
\usepackage{amssymb}
\usepackage{amsthm}
\usepackage{amsfonts}
\usepackage{enumerate}
\usepackage[scriptsize,hang,raggedright]{subfigure}
\usepackage[cmyk]{xcolor}

\usepackage{hyperref}
\usepackage{pdfsync}
\usepackage{dsfont}
\usepackage{color}

\usepackage{cite}
\usepackage{bbm}
\usepackage{multirow}

\numberwithin{equation}{section} 

\newtheorem{theorem}{Theorem}[section]
\newtheorem{proposition}[theorem]{Proposition}
\newtheorem{lemma}[theorem]{Lemma}
\newtheorem{corollary}[theorem]{Corollary}

\newtheorem{defi}[theorem]{Definition}
\theoremstyle{remark}
\newtheorem{remark}[theorem]{Remark}

\newcommand{\R}{\mathbb{R}}
\newcommand{\disp}{\displaystyle}

\newcommand{\ba}{\begin{array}}
\newcommand{\ea}{\end{array}}

\newcommand{\bthm}{\begin{theorem}}
\newcommand{\ethm}{\end{theorem}}
\newcommand{\bprop}{\begin{proposition}}
\newcommand{\eprop}{\end{proposition}}
\newcommand{\blemma}{\begin{lemma}}
\newcommand{\elemma}{\end{lemma}}

\newcommand{\beqn}{\begin{equation}}
\newcommand{\eeqn}{\end{equation}}
\newcommand{\beqns}{\begin{equation*}}
\newcommand{\eeqns}{\end{equation*}}

\newcommand{\supp}{\operatorname{supp}}

\newcommand{\pr}{\prime}
\newcommand{\pt}{\partial}
\newcommand{\arrow}{\rightarrow}

\renewcommand{\leq}{\leqslant}
\renewcommand{\geq}{\geqslant}

\definecolor{mygreen}{rgb}{0.1,0.75,0.2}

\newcommand{\eps}{\epsilon}

\newcommand{\E}{\mathsf{E}}
\newcommand{\Ea}{\mathsf{E}^a}
\newcommand{\Er}{\mathsf{E}^r}
\newcommand{\F}{\mathsf{F}}

\newcommand{\Rd}{{\mathord{\mathbb R}^d}}
\newcommand{\loc}{{\rm loc}}

\newcommand{\C}{\mathcal{C}}

\def\P{{\mathcal P}}
\def\S{{\mathcal S}}

\title[Regularized Nonlocal Interaction Energies]{Convergence of Regularized Nonlocal Interaction Energies}

\author{Katy Craig}
\address{Department of Mathematics, University of California, Santa Barbara, CA}
\email{kcraig@math.ucsb.edu}
\author{Ihsan Topaloglu}
\address{Department of Mathematics and Statistics, McMaster University, Hamilton, ON}
\email{itopalog@math.mcmaster.ca}
\thanks{KC's research is supported by U.S. National Science Foundation grant DMS 1401867 and the UC President's Postdoctoral Fellowship. IT's research is supported by a Fields--Ontario Postdoctoral Fellowship.}

\date{\today}                                        

\begin{document}

\begin{abstract}
Inspired by numerical studies of the aggregation equation, we study the effect of regularization on nonlocal interaction energies. We consider energies defined via a repulsive-attractive interaction kernel, regularized by convolution with a mollifier. We prove that, with respect to the 2-Wasserstein metric, the regularized energies $\Gamma$-converge to the unregularized energy and minimizers converge to minimizers. We then apply our results to prove $\Gamma$-convergence of the gradient flows, when restricted to the space of measures with bounded density.\end{abstract}

\maketitle

\section{Introduction}\label{sec:intro} 

We consider the nonlocal interaction energy
	\beqn \label{interactionenergy} 
		\E(\mu) = \int_{\Rd}\!\int_{\Rd} K(x-y)\,d\mu(x) d\mu(y),
	\eeqn
over the space $\P_2(\Rd)$ of probability measures with finite second moment, where the pairwise interaction kernel $K: \Rd \to \R\cup\{+\infty\}$ is an even, locally integrable, and lower semicontinuous function. Depending on the choice of $K$, the asymptotic states of many physical and biological systems can be characterized as minimizers of this energy. Of particular interest are interaction kernels which are repulsive at short distances and attractive at long distances. Important examples of such repulsive-attractive kernels are Morse-type potentials
\beqn
\begin{gathered} \label{repulsive attractive morse}
K(x) = C_r e^{-|x|/l_r} - C_a c^{-|x|/l_a}\\
\text{with } C_r/C_a < \left( l_r/l_a \right)^{-d},\ 0< l_r < l_a, \text{ and } 0<C_a<C_r,
\end{gathered}
\eeqn
and potentials in the power-law form
\beqn\label{repulsive attractive power law}
K(x) = |x|^q/q - |x|^p/p \quad \text{with} \quad -d<p<q,
\eeqn
which arise in models of granular media \cite{BenedettoCagliotiCarrilloPulvirenti, LiToscani, CarrilloMcCannVillani, CarrilloMcCannVillani2}, molecular self-assembly \cite{DoyeWalesBerry, Wales, RechtsmanStillingerTorquato}, biological swarming \cite{TopazBertozzi2, BernoffTopaz}, and the distributions of eigenvalues for Gaussian random matrices \cite{PetzHiai, ChafaiGozlanZitt}.

Due to the competition between the repulsive and attractive terms, the minimization of these energies leads to complex equilibrium configurations \cite{BalagueCarrilloLaurentRaoul,Bertozzietal_RingPatterns, BertozziLaurentLeger, BertozziGarnettLaurent, BertozziBrandman, BertozziCarrilloLaurent, Dong, FellnerRaoul2, FetecauHuangKolokolnikov, FetecauHuang, HuangBertozzi2, Poupaud, SunUminskyBertozzi, TopazBertozzi1}. In the case of repulsive-attractive power law kernels (\ref{repulsive attractive power law}), local minimizers  exhibit a variety of qualitatively different patterns, from solid rings to broken, rounded triangles \cite{Kolokolnikovetal_StabilityRingPatterns}, and the dimension of the minimizers' support can be characterized in terms of the strength of the repulsive forces \cite{BalagueCarrilloLaurentRaoul_Dimensionality}.

In addition to qualitative properties of local minimizers, there has also been significant interest in the existence and uniqueness of global minimizers. (Since (\ref{interactionenergy}) is translation invariant, minimizers are only unique up to translation.) Existence was recently established for a range of kernels, including (\ref{repulsive attractive power law}), by Simione, Slep\v{c}ev, and the second author \cite{SiSlTo2014} and by Ca\~nizo, Carrillo, and Patacchini \cite{CCP}. On the other hand, characterizing which kernels have a unique global minimizer is essentially open, due to lack of convexity.  In the particular case of Coulomb repulsion and quadratic attraction (when the energy is convex with respect to $H^{-1}$) the unique global minimizer is the indicator function on a ball (c.f. for example \cite{CDM14, ChFeTo14}).

Alongside the static problem of characterizing the minimizers of the interaction energy, there has also been significant interest in the dynamic problem of understanding the behavior of systems as they evolve toward a local minimizer. Such systems arise in biological swarming \cite{MogilnerEdelstein,MogilnerEdelsteinBent, TopazBertozzi1}, robotic swarming \cite{ChuangHuangDorsognaBertozzi, PereaGomezElosegui}, and assembly of viral capsid proteins \cite{HaganChandler}, and they may be modeled as \emph{gradient flows} of the energy with respect to the 2-\emph{Wasserstein metric}. Formally, this corresponds to the nonlinear, nonlocal partial differential equation 
\begin{align} \label{agg eqn}
\begin{cases}
\rho_t + \nabla \cdot (v \rho) = 0  \quad \text{with} \quad v = - 2 \nabla K * \rho, \\
\rho(x,0) = \rho_0(x),
\end{cases}
\end{align}
known as the \emph{aggregation equation}. (We choose to put a factor of two in the velocity field instead of putting a factor of one half on the energy (\ref{interactionenergy}).) For semi-convex interaction kernels $K$, with up to a Lipschitz singularity, weak measures solutions to  (\ref{agg eqn}) exist for all time and are unique \cite{AGS, 5person}. If the kernels are sufficiently convex,  solutions converge exponentially to a unique global minimizer \cite{CarrilloMcCannVillani}.  However, for nonconvex kernels with merely integrable singularities, much less is known about the evolution of measure solutions and their asymptotic behavior. In the particular case of Coulomb repulsion and quadratic attraction, solutions with bounded, continuous initial data converge to the unique global minimizer algebraically in time \cite{BertozziLaurentLeger}.

Due to the analytical difficulties repulsive-attractive kernels present, both theoretical and applied work is often complemented by numerical simulations. The most common method for simulating (\ref{agg eqn}) is a particle method, in which one approximates the initial data by a sum of Dirac masses, $\rho_0 \approx \sum_{i=1}^N \delta_{x_i} m_i$, where $x_i$ is the location of the Dirac mass and $m_i \geq 0$ is its weight. The corresponding solutions of (\ref{agg eqn}) are formally of the form $\rho \approx \sum_{i=1}^N \delta_{X_i(t)} m_i$, where the trajectories of the Dirac masses $X_i(t)$ satisfy the following finite system of ODEs:
\begin{align} \label{particleODEsystem}
 \frac{d}{dt} X_i(t) = - 2 \sum_{j=1}^N \nabla K(X_i(t) - X_j(t)) m_i , \quad X_i(0) = x_i.
 \end{align}
  For a range of interaction kernels, including repulsive-attractive power-laws (\ref{repulsive attractive power law}) in the parameter regime $2-d< p \leq 2, \  q>0$, Carrillo, Choi, and Hauray prove the convergence of the particle method solution $\sum_{i=1}^N \delta_{X_i(t)} m_i$ to weak measure solutions of the aggregation equation  \cite{CarrilloChoiHauray}.
  
  In recent work, Bertozzi and the first author consider a modification of these types of particle methods \cite{CrBe14}, analogous to classical vortex blob methods from fluid dynamics (c.f. \cite{BealeMajda1,BealeMajda2,AndersonGreengard} and references therein). Specifically, since the gradient of the kernel $\nabla K$ may be singular at the origin, they regularize the interaction kernel by convolution with a mollifier, leading to an approximate solution of the form $\sum_{i=1}^N \delta_{X^\epsilon_i(t)} m_i$, where $X^\epsilon_i(t)$ satisfies (\ref{particleODEsystem}) with $\nabla K$ replaced by $\nabla (K*\varphi_\epsilon)$. By regularizing the interaction kernel $K$, the authors extend particle methods to a wide range of interaction kernels, including the power-law kernels (\ref{repulsive attractive power law}) for $2-d\leq p \leq q$,  and obtain quantitative rates of convergence to classical solutions. While this convergence result is limited to bounded time intervals, numerical results indicate that regularizing the energy in this way may also be useful in studying asymptotic behavior (see Figure \ref{fig:only_one}).
\begin{figure}[htbp]
\begin{center}
\includegraphics[trim={0cm -.07cm 0cm 0cm},clip,width=.24\textwidth]{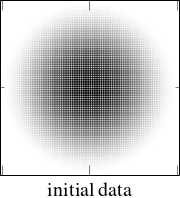} 
\includegraphics[trim={0cm 0cm 0cm 0cm},clip,width=.24\textwidth]{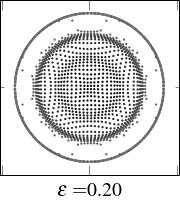} 
\includegraphics[trim={0cm 0cm 0cm .02cm},clip,width=.24\textwidth]{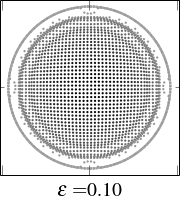} 
\includegraphics[trim={0cm -.08cm 0cm 0cm},clip,width=.24\textwidth]{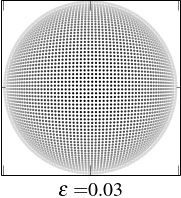}

\caption{Numerical solutions of (\ref{agg eqn}) at $t=100$ for Coulomb repulsion and quadratic attraction in two dimensions. The initial data is the function $\rho_0(x,y) = C(1-x^2-y^2)_+^2$, 
with $C$ chosen so $\rho_0$ has mass one.
Solutions are generated by the blob method \cite{CrBe14} with $\epsilon = 0.20, 0.10,$ and $0.03$ on the square $[-1,1]\times[-1,1]$, which is discretized on a $80\times 80$ grid. As $\epsilon\to 0$, the support of the numerical solution spreads evenly throughout the ball. This suggests that, in the limit, the solution converges to the indicator function on the ball,  the unique steady state of the PDE.}
\label{fig:only_one}
\end{center}
\end{figure}

Inspired by these results, we consider a regularization of the nonlocal interaction energy  analogous to Bertozzi and the first author's blob method. 
Specifically, given a smooth, radial, rapidly decreasing mollifier $\varphi$, we define the regularized kernel
		\[ 
			K_\epsilon := \varphi_\epsilon * K *\varphi_\epsilon,
		\]
with $\varphi_\epsilon(x) := \epsilon^{-d} \varphi(x/\epsilon)$, and consider the regularized interaction energy
  \beqn
		\label{reginteractionenergy} 
			\E_\epsilon(\mu) := \int_{\Rd}\!\int_{\Rd} K_\epsilon(x-y)\,d\mu(x) d\mu(y).
	\eeqn
We show that, for a range of repulsive-attractive kernels, including repulsive-attractive power-law kernels (\ref{repulsive attractive power law}) with $2-d \leq p< 0< q \leq 2$,
the regularized energies $\E_\epsilon$ $\Gamma$-converge to $\E$ with respect to the 2-Wasserstein metric. Then, via a compactness argument, we show that, up to a subsequence, the minimizers of $\E_\epsilon$ converge to a minimizer of $\E$ as $\epsilon \to 0$. This provides a method of approximating minimizers of $\E$ by minimizers of energies with superior convexity and regularity properties (see Remarks \ref{rem:lambdaconv} and \ref{rem:differentiability}). It also demonstrates a type of continuity among minimizers of $\E_\epsilon$, in spite of their vastly different regularity properties, as illustrated in Figure \ref{fig:only_one}. Finally, it imparts further theoretical justification for the success of the Bertozzi and the first author's numerical blob method. Our work builds on previous work of Fellner and Raoul \cite{FellnerRaoul2}, who considered regularization of the Coulomb repulsion in one dimension, showing that, in the presence of a confining potential, Dirac-type stationary states converge to a unique $L^\infty$ stationary state.

We conclude by considering the convergence of the corresponding gradient flows for repulsive-attractive power-law kernels \eqref{repulsive attractive power law} in the regime $2-d \leq p< 0< q \leq 2$. Indeed, using the abstract scheme of Serfaty \cite{Serfaty}, we show that as $\eps\to 0$ gradient flows of $\E_\epsilon$ that are bounded in $L^\infty(\Rd)$ $\Gamma$-converge to a generalized notion of gradient flow for $\E$, that is to its \emph{curve of maximal slope} on the metric space of probability measures with bounded density. This provides a link between the well-understood case of gradient flows of semi-convex energies $\E_\epsilon$ and the curve of maximal slope of the unregularized energy $\E$. It also provides a first step in understanding the connection between the gradient flows of unregularized and noncovex interaction energies $\E$ and the aggregation equation, via a singular perturbation approach (see Remark \ref{rem:grad_flow_agg_eqn}).

These results provide many directions for future work. A natural extension would be to consider the effect of simultaneously discretizing and regularizing the energy and to  study the distinguished limits as the discretization and regularization are removed. Another interesting direction would be to seek quantitative estimates on the rate of convergence of minimizers and gradient flows as a regularization is removed. Our present work strongly leverages compactness arguments, and we believe that a very different approach would be needed in order to provide quantitative estimates.

\medskip

\begin{remark}[Power-law kernels and generalizations]\label{rem:powerlaw_generalization} 
{\rm 
In sections \ref{sec:prelim} and \ref{sec:conv}, we choose to state our results for interaction kernels in the power-law form \eqref{repulsive attractive power law}, not only because these kernels have attracted much interest in connection with the aggregation equation (cf. \cite{BertozziLaurentRosado,Dong,FetecauHuang,FetecauHuangKolokolnikov,YaoBertozzi}) and nonlocal interaction energies (cf. \cite{CaChHu,CDM14,ChFeTo14}), but also because these kernels capture the crucial behavior of short-range and long-range interactions, i.e., at the origin and at infinity. Perhaps one of the most important properties of such kernels, which we exploit in our analysis, is that the attractive and repulsive parts of the kernel are given separately by relatively simple potentials. This allows us to quantitatively compare the attractive and repulsive parts of the energy. Indeed, it is due to the comparison (or rather competition) between these two terms that minimization of such energies exhibit energy-driven pattern formation.

Our analysis depends mostly on the behavior of $K$ for very large and very small values of $|x|$; hence, there are many possible ways of extending our results in sections \ref{sec:prelim} and \ref{sec:conv} to a wider class of kernels. In section \ref{sec:generalization}, we provide general conditions on the regularity, growth at infinity, and behavior around singularities of $K$ for which our results continue to hold. In particular, our results hold in any dimension $d \geq 1$ for kernels that exhibit power-law repulsion at the origin and power-law attraction at infinity, including logarithmic repulsion in $d \geq 2$.
}
\end{remark}

\medskip

\begin{remark}[Choices of parameters in power-law kernels] \label{rem:powerlawchoices}
{\rm 
In sections \ref{sec:prelim} and \ref{sec:conv}, we restrict our attention to repulsive-attractive power law kernels satisfying (\ref{repulsive attractive power law}) with $2-d \leq p < 0 < q\leq 2$ in dimension $d \geq 3$. We require $q \leq 2$ so that the attractive term and its regularization are functions with at most quadratic growth and their convolution with $\mu \in \P_2(\Rd)$ is finitely valued (see Remark \ref{rem:attraction powers}). We require $q>0$ to obtain compactness of the sequence of minimizers corresponding to the regularized energies $\E_\eps$ (see Remark \ref{rem:negative_attr_power}). We take $p < 0$ since our results are trivially true for nonsingular attractive-repulsive power-law kernels (see section \ref{sec:generalization}). On the other hand, we take $2-d\leq p$ in order to ensure monotonicity of the repulsive part of the regularized kernels with respect to $\eps$, when the mollifier is the heat kernel (see Proposition \ref{prop:monoton}).}
\end{remark}

\medskip

Our paper is organized as follows. In section \ref{sec:prelim}, we recall fundamental results on the 2-Wasserstein metric and the regularization of probability measures. We use these to establish the above mentioned monotonicity property of the regularized energies (\ref{reginteractionenergy}) in the regularization parameter $\eps>0$. In section \ref{sec:conv}, we prove the $\Gamma$-convergence of $\E_\epsilon$ to $\E$ and the corresponding convergence of minimizers. In section \ref{sec:generalization}, we extend these results to more general interaction kernels. Finally, in section \ref{gradflowsection}, we introduce background on curves of maximal slope and prove the $\Gamma$-convergence of these generalized gradient flows for repulsive-attractive power-law kernels restricted to bounded densities.
 
 Throughout, we denote by $L^p(\Rd)$ the space of functions whose $p$-th power is integrable with respect to the Lebesgue measure on $\Rd$ and by $L^p(\mu)$ those functions whose $p$-th power is integrable with respect to the probability measure $\mu\in\P(\Rd)$. $C(\Rd)$ denotes the space of continuous functions on $\Rd$, $C^\infty(\Rd)$ denotes the space of smooth functions on $\Rd$, and $B_r(x)$ denotes the open ball of radius $r$ centered at $x\in\Rd$. All constants which appear in the text are positive and may change from line to line. We express the dependence of constants to parameters by a subscript such as $C_{d,p}$ when necessary.

\section{Preliminaries}\label{sec:prelim}

\subsection{The 2-Wasserstein metric and regularization of probability measures}
We consider the energies $\E_\epsilon$ and $\E$ over the space
	\[
		\P_2(\Rd) := \left\{ \mu\in\P(\Rd) \colon \int_{\Rd} |x|^2\,d\mu(x) < +\infty \right\}
	\] 
of probability measures with finite second moment. We endow this space with the 2-Wasserstein metric, which we recall briefly now, along with some of its basic properties. For further background, we refer the reader to the books by Ambrosio, Gigli and Savar\'e \cite{AGS} and Villani \cite{Villani}.

The 2-Wasserstein distance between $\mu, \nu \in \P_2 (\Rd)$ is 
	\beqn\label{eqn:Wass_dist}
		d_W(\mu,\nu):=\left(\min\left\{\int_{\Rd}\!\int_{\Rd}|x-y|^2\,d\gamma(x,y)\colon \gamma\in\C(\mu,\nu)\right\}\right)^{1/2},
	\eeqn
where $\C(\mu,\nu)$ is the set of transport plans between $\mu$ and $\nu$,	\[
		\C(\mu,\nu):=\left\{\gamma\in\P(\Rd\times\Rd)\colon (\pi_1)_\#\gamma=\mu\quad\text{and}\quad(\pi_2)_\#\gamma=\nu\right\}.
	\]
Here $\pi_1$, $\pi_2$ denote the projections $\pi_1(x,y)=x$ and $\pi_2(x,y)=y$. For $i=1,2$, $(\pi_i)_{\#}\gamma$ denotes the pushforward of $\gamma$ defined by $(\pi_i)_{\#}\gamma (U):=\gamma(\pi_i^{-1}(U))$ for any measurable set $U\subset\Rd$.

The minimization problem \eqref{eqn:Wass_dist} admits a solution, i.e., there exists an optimal transport plan $\gamma_0\in \C_0(\mu,\nu)$ so that
	\[
		d_W^2(\mu,\nu)=\int_{\Rd}\!\int_{\Rd} |x-y|^2\,d\gamma_0(x,y).
	\]
Moreover, $(\P_2(\Rd),d_W)$ is a complete and separable metric space, and convergence in $(\P_2(\Rd),d_W)$ can be characterized as follows:
\smallskip
\begin{center}
	\begin{tabular}{lcl}
		$d_W(\mu_n,\mu)\arrow 0$ & $\Longleftrightarrow$ & $\mu_n\arrow\mu$ weak-$*$ in $\P(\Rd)$ and $\int_{\Rd}|x|^2\,d\mu_n(x)\arrow\int_{\Rd}|x|^2\,d\mu(x)$,
		\\
		\\
		& $\Longleftrightarrow$  & $\int_{\Rd} f(x)\,d\mu_n(x) \to \int_{\Rd} f(x) \,d\mu(x)$,\\[0.25cm]
		&   & $\forall f\in C(\Rd)$ such that $|f(x)| \leq C(1+ |x-x_0|^2)$.

	\end{tabular}
\end{center}
\medskip
We will refer to functions satisfying $|f(x)| \leq C(1+ |x-x_0|^2)$, for some $C>0$ and $x_0 \in \Rd$, as functions with \emph{at most quadratic growth}.

\bigskip

In addition to regularizing our energy functionals by convolution with a mollifier, we will also regularize our measures. We recall the definition of the convolution of a measure.

\medskip

\begin{defi}
For $\mu \in\P_2(\Rd)$ and $\varphi$ continuous with at most quadratic growth, $\mu * \varphi$ is defined by
	\begin{align} \label{convdef}
		\int_{\Rd} f(x)\,d(\mu* \varphi)(x) =  \int_{\Rd} f*\varphi(y)  d\mu(y),
	\end{align}
for all bounded measurable functions $f:\Rd\arrow\R$.
\end{defi}

\medskip

\begin{remark}
{\rm
Note that $\mu * \varphi$ is absolutely continuous with respect to Lebesgue measure and $d (\mu * \varphi) = \mu*\varphi(x)\,dx$, where $\mu*\varphi(x)= \int_{\Rd} \varphi(x-y)\,d\mu(y)$. If, in addition,  $\varphi(x) \in L^\infty(\Rd)$, then $\mu* \varphi(x) \in L^\infty(\Rd)$ and (\ref{convdef}) holds for all $f \in L^1(\Rd)$.
}
\end{remark}

\medskip

We recall the following lemma on the approximation of measures by convolution.

\medskip

\begin{lemma}[cf.{\cite[Lemma 7.1.10]{AGS}}]\label{lem:AGS_appr_conv}
	Fix $\mu\in\P_2(\Rd)$ and a mollifier $\varphi \in C^{\infty}(\Rd)$ with finite second moment, $\varphi \geq 0$, and $\int_{\Rd} \varphi(x) dx = 1$.
Then for $\varphi_\epsilon(x):=\epsilon^{-d}\varphi(x/\epsilon)$, $\mu*\varphi_\eps \in \P_2(\Rd)$ and
		\beqn\label{eqn:mollifieraprx}
			d_W(\mu*\varphi_\eps,\mu) \leq \eps\,\left(\int_{\Rd}|x|^2\varphi(x)\,dx\right)^{1/2}.
		\eeqn
\end{lemma}

\subsection{Regularization of energies} \label{regenergysec}
In what follows, we consider mollifiers which satisfy the following assumptions:
\medskip
		\begin{itemize}	
				\addtolength{\itemsep}{6pt}
						\item[(M1)] $\varphi \in C^\infty(\Rd)$  and $\varphi(x)=\varphi(|x|)\geq 0$ for all $x\in\Rd$,
						\item[(M2)] $\int_{\Rd} \varphi(x)\,dx=1$,
						\item[(M3)] $\varphi(x)\leq C|x|^{-l}$ for some $l\geq \max \{2d+1,4 \}$ and $C>0$.
		\end{itemize}
		\medskip
		\begin{remark}[Finite second moment of mollifiers]\label{rem:moll_fin_sec_mom}
{\rm
Any mollifier $\varphi$ satisfying the assumptions (M1)--(M3) has finite second moment, i.e., $\int_{\Rd} |x|^2\varphi(x)\,dx<+\infty$.
}
\end{remark}

\medskip

As a consequence of the previous results on regularization of measures in the Wasserstein metric,
we have the following lemma which relates $\E_\epsilon$ to $\E$.

\medskip

\begin{lemma} \label{moveepsilontomu}
Consider an interaction energy $\E(\mu) = \iint_{\Rd \times \Rd} K(x-y)\,d\mu(x) d\mu(y)$,
where $K(x)$ is an even, nonnegative, locally integrable function that has most quadratic growth and is continuous outside of some ball. Given a mollifier $\varphi$ satisfying (M1)--(M3), then for all $\mu\in\P_2(\Rd)$
	\[
		\E_\epsilon(\mu) = \E(\mu*\varphi_\epsilon) .\]
\end{lemma}

\begin{proof}
Define the right shift operator $\tau_y: f(x) \mapsto f(x-y)$. For any $f$ and $g$ even,
	\begin{align*}
			\tau_y[f*g](x) &= f*g(x-y) = \int_{\Rd} f(x-y-z) g(z)\,dz \\
			&=  \int_{\Rd} \tau_y[f](x-z) g(z)\,dz = \tau_y[f] * g(x).
  \end{align*}
Furthermore, since $f*g(x-y) = f*g(y-x)$, we also have $\tau_y[f*g](x) = \tau_x[f]*g(y)$. 
Since $\varphi \in L^\infty(\Rd)$, (\ref{convdef}) holds for all continuous $f$ of at most quadratic growth and for all $f \in L^1(\Rd)$. As we may decompose $K$ into the sum of a continuous function of at most quadratic growth (away from the origin) and an integrable component (near the origin), the result then follows:
	\begin{align*}
		\E_\epsilon(\mu) &= \int_{\Rd}\!\int_{\Rd} \varphi_\epsilon * K * \varphi_\epsilon(x-y)\, d\mu(x) d\mu(y) = \int_{\Rd}\!\int_{\Rd} \tau_y[\varphi_\epsilon * K] * \varphi_\epsilon(x)\, d\mu(x) d\mu(y), \\
&= \int_{\Rd}\!\int_{\Rd} \varphi_\epsilon * \tau_y[K] * \varphi_\epsilon(x)\, d\mu(x) d\mu(y) = \int_{\Rd}\!\int_{\Rd} \varphi_\epsilon * \tau_y[K](x) \,d[\mu*\varphi_\epsilon](x) d\mu(y),   \\
&= \int_{\Rd}\!\int_{\Rd} \varphi_\epsilon * \tau_x[K](y)\, d[\mu*\varphi_\epsilon](x) d\mu(y)  = \int_{\Rd}\!\int_{\Rd}  K(x-y)\, d[\mu*\varphi_\epsilon](y) d[\mu*\varphi_\epsilon](x), \\
&= \E(\mu*\varphi_\epsilon).
\end{align*}
\end{proof}

\medskip

For general interaction kernels and mollifiers, there is no uniform relation between the size of $\E_\epsilon(\mu)$ compared to  $\E(\mu)$. 
Consequently, in this section and section \ref{sec:conv}, we consider interaction energies of the following form:	
\medskip
		\begin{itemize}	
				\addtolength{\itemsep}{6pt}
						\item[(E1)] $\E(\mu) = \iint_{\Rd \times \Rd} K(x-y)\,d\mu(x) d\mu(y)$ for $K(x) = |x|^q/q - |x|^p/p$,
						\item[(E2)] $2-d \leq p < 0< q\leq 2$.
		\end{itemize}	
\medskip
If we mollify energies satisfying (E1)--(E2) via the heat kernel, we obtain monotonicity in $\eps$ of the repulsive part of the interaction energy. This generalizes a property used by Blanchet, Carlen, and Carrillo \cite{BCC} for the Newtonian potential.
In section \ref{sec:generalization}, we describe how to extend these results to more general interaction energies.

Let $K^a:=\frac{1}{q}|x|^q$ denote the attractive part of the interaction kernel and $K^r:=-\frac{1}{p}|x|^p$ the repulsive part. Similarly, let $\Ea$ denote the attractive part of the energy and $\Er$ denote the repulsive part,
	\beqn \label{attractiverepulsivepartenergy}	
		\Ea(\mu) = \int_{\Rd}\!\int_{\Rd} K^a(x-y)\,d\mu(x)d\mu(y), \quad
		\Er(\mu) = \int_{\Rd}\!\int_{\Rd} K^r(x-y)\,d\mu(x)d\mu(y).	
	\eeqn

\medskip

\bprop[Monotonicity of $\E^r_\epsilon$]\label{prop:monoton}
Suppose the energy $\E$ satisfies (E1)--(E2) and the regularization mollifier $\psi_\eps$ is given by the heat kernel. Then for all $\epsilon_1\geq\epsilon_2 > 0$ and $\mu \in \P_2(\Rd)$
		\[
			\E^r_{\epsilon_1}(\mu) \leq \E^r_{\epsilon_2}(\mu).
	\]
\eprop

\begin{proof}
Since $\psi_\eps(x)=\psi(\eps,x)=(4\pi\eps)^{-d/2}e^{-|x|^2/4\eps}$, $\pt_\eps\psi=\Delta_x\psi$ for all $\eps>0$ and $x\in\Rd$. In the parameter regime $2-d < p <0$, the Riesz potential
	\[
		I_{d+p}(f)(x)= C_{d,p} \int_{\Rd} f(y)|x-y|^p\,dy = - C_{d,p}\,p K^r*f(x) 
	\]
satisfies the identity
	\[
		I_{d+p}(\Delta f) = \Delta I_{d+p}(f) = -I_{d+p-2}(f)
	\]
for any Schwartz class function $f$ (see e.g. \cite{SteinSingularIntegrals}). Moreover, $I_{d+p}(f)\geq 0$ for any $f\geq 0$.

Defining $K^r_\epsilon(x) = \psi_\eps*K^r * \psi_\eps$ and using the above identity we obtain
	\beqn
		\begin{aligned}
	 \frac{\pt}{\pt \epsilon}K^r_\epsilon(x) &= 2\psi_\eps * K^r * \left(\partial_\eps \psi_\eps \right) 
	 																				 = 2\psi_\eps * K^r * (\Delta \psi_\epsilon) = -(p\,C_{d,p})^{-1}\psi_\eps * I_{d+p}(\Delta \psi_\eps)\\
	 																				 &= (p\,C_{d,p})^{-1}\,\psi_\eps*I_{p+d-2}(\psi_\eps) \leq 0
	 	\end{aligned}
	 	\nonumber
  \eeqn
since $p<0$ and $\psi_\eps\geq 0$. Thus, $K^r_\epsilon$ is monotonically decreasing in $\epsilon$.

When $p=2-d$, $I_{d+p}(f)$ is the Newtonian potential and satisfies the identity $I_{d+p}(\Delta f) = \Delta I_{d+p}(f) = -f$. Again, differentiating $K^r_\eps$ with respect to $\eps$ yields the desired result. Therefore, by definition of  $\E^r$ \eqref{attractiverepulsivepartenergy}, we conclude  monotonicity in $\eps$.
\end{proof}

\medskip

We close this section with brief remarks on the convexity and differentiability of regularized energies.

\medskip

\begin{remark}[$\lambda_\epsilon$-convexity of $\E_\epsilon$] \label{rem:lambdaconv}
{\rm
 Recall that a function $f:\Rd\arrow\R\cup\{+\infty\}$ is $\lambda$-convex if $f(x)-\frac{\lambda}{2}|x|^2$ is convex for some $\lambda \leq 0$. If $f$ is twice continuously differentiable, this is equivalent to $D^2 f(x) \geq \lambda I $ for all $x \in \Rd$. Consequently, if $\varphi$ satisfies (M1)--(M3) and $\E$ satisfies (E1)--(E2), then  both the regularized potentials $K_\epsilon$ and their opposites $-K_\epsilon$ are $\lambda_\epsilon$-convex with $\lambda_\epsilon = -C_\varphi\epsilon^{-d}$, where $C_\varphi$ is a positive constant depending on $\varphi$ and $\eps$ is sufficiently small. For an interaction energy of the form (\ref{interactionenergy}), $\lambda$-convexity of the kernel $K$ ensures $\lambda$-convexity of the energy $\E$ with respect to the Wasserstein metric \cite{5person}. Thus, $\E_\epsilon$ and $-\E_\epsilon$ are both $\lambda_\epsilon$-convex.
 }
\end{remark}

\medskip

\begin{remark}[Differentiability of $\E_\epsilon$] \label{rem:differentiability}
{\rm
Given a functional $\F: \P_{2}(\Rd) \to \R \cup \{+\infty\}$ and $\mu\in\P_2(\Rd)$ so that $\F(\mu)< +\infty$,  the \emph{metric local slope} of $\F$ at $\mu$ is
\[ |\partial \F|(\mu) := \limsup_{\nu \to \mu} \frac{(\F(\mu) - \F(\nu))_+}{d_W(\mu,\nu)}. \]
If $\varphi$ satisfies (M1)--(M3) and $\E$ satisfies (E1)--(E2),  the metric local slope for  $\E_\epsilon$ is well-defined (cf. \cite[Proposition 2.2]{5person}, \cite[Lemma 10.1.5]{AGS}) and
 \[ | \partial \E_\epsilon|(\mu) = 2\|\nabla K_\epsilon *\mu \|_{L^2(\mu)}. \]

 Furthermore, for $\mu \in \P_2(\Rd)$, we have $2\|\nabla K_\epsilon *\mu \|_{L^2(\mu)} \leq C_{\varphi, \mu}  \epsilon^{1-d} $ for $\eps$ sufficiently small. 
 To see this, let $0 \leq \eta(x) \leq 1$ be a smooth function satisfying $\eta(x) = 0$ if $|x| \leq 1$ and $\eta(x) = 1$ if $|x| \geq 2$. Using $\eta$, we decompose $\nabla K_\epsilon$ into its singular and nonsingular components, $\nabla K_\epsilon = \nabla K_\epsilon (1-\eta) + \nabla K_\epsilon \eta = \nabla K_\epsilon^s + \nabla K_\epsilon^n$. By linearity of convolution and the triangle inequality for $L^2(\mu)$, it suffices to estimate $\|\nabla K_\epsilon^s* \mu\|_{L^2(\mu)}$ and $\|\nabla K_\epsilon^n*\mu\|_{L^2(\mu)}$. The former is bounded by $\|\nabla K^s_\epsilon * \mu\|_{L^\infty(d \mu)} \leq \sup_{x \in \Rd} |\nabla K^s_\epsilon(x)| \leq C_\varphi \epsilon^{1-d}$. To bound the latter, we apply Minkowski's integral inequality, the fact that $\nabla K^n_\epsilon$ has at most linear growth, and the fact that $\mu$ has finite second moment to obtain 
 \[ \|\nabla K_\epsilon^n*\mu\|_{L^2(\mu)} \leq \int_{\Rd} \left( \int_{\Rd} |\nabla K_\epsilon^n(x-y)|^2 \,d \mu(y) \right)^{1/2} d\mu(x) \leq C_{\varphi, \mu}.\]
}
\end{remark}

\medskip

\begin{remark}[Attraction powers $q>2$] \label{rem:attraction powers}
{\rm
 One possible way to extend our results to cover attraction powers $q > 2$ would be to consider the energy $\E$ over the space $\P_q(\Rd)$ of probability measures with finite moments up to order $q$, endowed with the $q$-Wasserstein distance
   \[
      d_{q}(\mu,\nu) := \left(\min\left\{\int_{\Rd}\!\int_{\Rd}|x-y|^q\,d\gamma(x,y)\colon \gamma\in\C(\mu,\nu)\right\}\right)^{1/q},
   \]
where $\C(\mu,\nu)$ is the set of transport plans between $\mu$ and $\nu$.
 Alternatively, to cover all  $q >0$, one could consider the $\infty$-Wasserstein metric
  \[
    d_{\infty}(\mu,\nu) := \inf_{\gamma\in\C(\mu,\nu)} \, \sup_{(x,y)\in\supp(\gamma)} |x-y|
  \]
 over the space $\P_{\infty}(\Rd)$ of probability measures with finite moments of all orders. The space $(\P_{\infty}(\Rd),d_{\infty})$ is a complete metric space \cite{GiSh84} and has been used in the study of local minimizers for several variational problems, including nonlocal repulsive-attractive energies \cite{BalagueCarrilloLaurentRaoul_Dimensionality,CDM14,McCann2}.

 The main difficulty in extending our results to $\P_q(\Rd)$ or $\P_{\infty}(\Rd)$ lies in understanding the appropriate generalization of the $\lambda$-convexity to these distances and, in particular, the appropriate analogue of the HWI inequality, which plays a key role in the proof of Theorem \ref{thm:gammaconv}. We leave this to future work.
}
\end{remark}

\section{Convergence of regularized energies and minimizers}\label{sec:conv}

We now turn to the proof that, up to a subsequence, minimizers of the regularized energies $\E_\epsilon$  converge to  a minimizer $\mu$ of $\E$ with respect to $d_W$. We establish this by first proving a $\Gamma$-convergence result for the sequence of energies $\E_\epsilon$ and then obtaining compactness for any sequence $\{\mu_\epsilon\}_{\epsilon>0}$ when $\E_\epsilon(\mu_\epsilon)$ is uniformly bounded. For the latter, we use Lions' result on concentration compactness, which we recall  for the readers' convenience. 

\medskip

\blemma[{Concentration compactness lemma for measures (cf. \cite{Lions84}, \cite[Section 4.3]{Struwe})}]\label{lem:conc_comp}
		Let $\{\mu_n\}_{n\in\mathbb{N}}$ be a sequence of probability measures on $\Rd$. Then there exists a subsequence $\{\mu_{n_k}\}_{k\in\mathbb{N}}$ satisfying one of the three following possibilities:
		\vspace{0.2cm}
	\begin{itemize}
		\item[(i)] \emph{(tightness up to translation)} There exists a sequence $\{y_k\}_{k\in\mathbb{N}}\subset\Rd$ such that for all $\epsilon>0$ there exists $R>0$ with the property that
			\[
				\int_{B_R(y_k)}\,d\mu_{n_k}(x) \geq 1-\epsilon \qquad {\hbox{\rm  for all $k$.}}
			\]
		 \item[(ii)] \emph{(vanishing)} $\disp \lim_{k\arrow\infty} \sup_{y\in\Rd} \int_{B_R(y)}\,d\mu_{n_k}(x)=0$, for all $R>0$; \vspace{0.2cm}
		 \item[(iii)] \emph{(dichotomy)} There exists $\alpha\in(0,1)$ such that for all $\epsilon>0$, there exist a number $R>0$ and a sequence $\{x_k\}_{k\in\mathbb{N}}\subset\Rd$ with the following property:\\
		 
		 	\vspace{-0.2cm}
		 Given any $R^\pr>R$ there are nonnegative measures $\mu_k^1$ and $\mu_k^2$ such that
		 	\vspace{0.2cm}
			\begin{itemize}
		 		\item[] $0\leq \mu_k^1 + \mu_k^2 \leq \mu_{n_k}$\,, \quad $\supp(\mu_k^1)\subset B_R(x_k)$,\quad $\supp(\mu_k^2)\subset \Rd\setminus B_{R^\pr}(x_k)$\,,
				\vspace{0.2cm}
		 		\item[] $\disp \limsup_{k\arrow\infty} \left(\left|\alpha-\int_{\Rd}d\mu_k^1(x)\right|+\left|(1-\alpha)-\int_{\Rd}d\mu_k^2(x)\right|\right)\leq \epsilon$.
		 	\end{itemize}
	\end{itemize}
\elemma

\bigskip

We begin our proof of the convergence of minimizers with the observation that minimizers exist for both the regularized and unregularized energies.

\medskip

\bprop[Existence of minimizers in $\P_2(\Rd)$]\label{prop:exist}
Suppose the energy $\E$ satisfies (E1)--(E2) and the mollifier $\varphi$ satisfies (M1)--(M3). Then both $\E$ and $\E_\epsilon$ admit a minimizer in $\P_2(\Rd)$, for $\epsilon >0$ sufficiently small.
\eprop

\begin{proof}
First note that, in the parameter regime $2-d < p<0<q\leq 2$, the interaction potential $K$ and its regularization $K_\epsilon$ are locally integrable and lower semicontinuous.

We now show that $K$ and $K_\epsilon$ are strictly increasing in each coordinate for $x_i$ sufficiently large. For  $K$ given as in (E1)--(E2) and $x \neq 0$,
		\[
		 \partial_{x_i}K(x) = x_i |x|^{q-2} - x_i |x|^{p-2} = x_i |x|^{q-2}(1- |x|^{p-q}).
	  \]
Consequently, if $x_i > 1$, $\partial_{x_i}K(x)$ is positive.

Now we consider $K_\epsilon$. It suffices to show that if $x_i >2$, then
		\begin{align} \label{regularizedincreasing1}
 			|\partial_{x_i} K_\epsilon(x) - \partial_{x_i} K(x) | \leq C \epsilon |x|^{q-2}
 		\end{align}
for some constant $C>0$. Then for $x_i >2$, we have  $(1- |x|^{p-q}) > \tilde{C}_{p,q}>0$ and
		\begin{align*}
 			\partial_{x_i} K_\epsilon(x) &= \partial_{x_i} K(x) + \partial_{x_i} K_\epsilon(x) - \partial_{x_i} K(x) \geq x_i |x|^{q-2}(1- |x|^{p-q}) - C \epsilon |x|^{q-2}\\
 &\geq 2|x|^{q-2}\tilde{C}_{p,q} - C \epsilon |x|^{q-2} = |x|^{q-2} (2\tilde{C}_{p,q} - C \epsilon).
 		\end{align*}
Choosing $\epsilon$ sufficiently small, this is nonnegative.
 
We now prove inequality (\ref{regularizedincreasing1}). First, we use a cutoff function to rewrite $\partial_{x_i} K(x)$ as the sum of a compactly supported singular function $\pt_{x_i}K^s$ and a continuously differentiable  function $\pt_{x_i}K^n$.  Let $0 \leq \eta \leq 1$ be a  smooth function satisfying $\eta(x) \equiv 0$ for $|x|< 1/4$ and $\eta(x) \equiv 1$ for $|x| > 1/2$. Write $\partial_{x_i} K =  (1-\eta) \partial_{x_i} K+ \eta \partial_{x_i} K  =: \pt_{x_i}K^s + \pt_{x_i}K^n$. Then, for $\Phi_{\eps}(y):=\int_{\Rd}\varphi_{\eps}(y-z)\varphi_{\eps}(z)\,dz$,
 		\begin{align*} 
 			&|\partial_{x_i} K_\epsilon(x) - \partial_{x_i} K(x)| = \left| \int_{\Rd} \left(\partial_{x_i} K(x-y) - \partial_{x_i} K(x) \right) \Phi_\epsilon(y)\,dy \right| \\
 																													&\quad \leq \left| \int_{\Rd} \left(\pt_{x_i}K^s(x-y) - \pt_{x_i}K^s(x) \right) \Phi_\epsilon(y)\,dy \right| \\
																				&\qquad + \left| \int_{\Rd} \left(\pt_{x_i}K^n(x-y) - \pt_{x_i}K^n(x) \right) \Phi_\epsilon(y)\,dy \right| \\
 																													 &\quad =: I_1 + I_2.
		\end{align*}
		
The autocorrelation function $\Phi_\eps$ satisfies $\Phi_\eps(y)=\eps^{-d}\Phi(y/\eps)$ for all $y\in\Rd$ and has the same decay property (M3) as $\varphi_\eps$. To see this, suppose $|y|=3R$ for some $R>0$. Then for any $z\in\Rd$ we have $|z|>R$ or $|y-z|>R$; hence,
		\beqn \nonumber
			\begin{aligned}
			|\Phi_\eps(y)| &= \frac{1}{\eps^{2d}} \int_{\Rd} \varphi\left(\frac{y-z}{\eps}\right)\varphi\left(\frac{z}{\eps}\right)\,dz \\
									   &\leq \frac{1}{\eps^{2d}} \left( \int_{|z|>R}\varphi\left(\frac{y-z}{\eps}\right)\varphi\left(\frac{z}{\eps}\right)\,dz + \int_{|y-z|>R} \varphi\left(\frac{y-z}{\eps}\right)\varphi\left(\frac{z}{\eps}\right)\,dz \right) \\
									   &\leq C\eps^{l-2d}R^{-l} = C\eps^{l-2d}|y|^{-l}.
			\end{aligned}
		\eeqn

First, we estimate $I_1$. If $x_i > 2$, then $|x| >2$, so $\pt_{x_i}K^s(x) = 0$. Consequently,
		\begin{align*}
				I_1= \left| \int_{|x-y|< 1} \pt_{x_i}K^s(x-y)  \Phi_\epsilon(y)\,dy \right|.
		\end{align*}
 Since $|x-y| <1$ and $|x| >2$,
		\[
			|y| \geq |x| - |y-x| \geq |x| - 1 \geq |x|/2.
		\]
Consequently,
		\[ 
			|\Phi_\epsilon(y)| \leq C\eps^{l-2d}|y|^{-l}  \leq C\, \epsilon\, |y|^{-l} \leq C\, 2^{l}\, \epsilon\, |x|^{-l}
		\]
for $0<\eps<1$.
Since $\pt_{x_i}K^s $ is locally integrable and $2-q <2d+1\leq l$, there exists $C>0$ so that for $|x| >2$,
		\[ 
			I_1 \leq C \,\epsilon\, |x|^{-l} \leq C\, \epsilon\, |x|^{q-2}.
		\]

Now we estimate $I_2$. Since $\pt_{x_i}K^n$ is continuously differentiable,
		\begin{align*}
			I_2 &=  \left| \int_{\Rd} \! \int_0^1 \frac{d}{d \alpha} \pt_{x_i}K^n(x- \alpha y)\,d \alpha \,\Phi_\epsilon(y)\, dy \right|  \\
			&=  \left| \int_{\Rd} \! \int_0^1 \langle \nabla \pt_{x_i}K^n(x- \alpha y), -y \rangle \, d \alpha \,\Phi_\epsilon(y)\, dy \right| \\
					&\leq \int_{|y| \leq |x|/2}  \int_0^1  |\nabla \pt_{x_i}K^n(x- \alpha y)|\, |y|\,d \alpha \, \Phi_\epsilon(y)\, dy \\ &\quad +  \int_{|y| > |x|/2}  \int_0^1  |\nabla \pt_{x_i}K^n(x- \alpha y)|\, |y|\,d \alpha \, \Phi_\epsilon(y)\, dy.
		\end{align*}
For  $|y| \leq |x|/2$, $|x - \alpha y| \geq |x|/2 > 1$ for all $\alpha \in [0,1]$. Thus,  $\nabla \eta(x-\alpha y) \equiv 0$, and
		\begin{align*}
			|\nabla \pt_{x_i}K^n(x- \alpha y) | \leq |(\nabla \eta(x- \alpha y) )\partial_{x_i} K(x- \alpha y) | + |\eta(x- \alpha y)  \nabla \partial_{x_i} K(x- \alpha y) | \leq C|x|^{q-2}.
		\end{align*} 
On the other hand, for all all $y \in \Rd$, $|\nabla \pt_{x_i}K^n(x-\alpha y)| \leq C$. Therefore,
		\begin{align*}
			I_2 &\leq C |x|^{q-2} \int_{|y| < |x|/2} |y|\, \Phi_\epsilon(y) \, dy + C \int_{|y| > |x|/2} |y|\, \Phi_\epsilon(y) \, dy \\
					&=  C \epsilon|x|^{q-2} \int_{|z| < |x|/(2\epsilon)}  |z| \,\Phi(z)\, dz +  C \epsilon \int_{|z| > |x|/(2 \epsilon)}  |z| \,\Phi(z)\, dz.
		\end{align*}
		
The first integral is bounded by a constant since $\Phi$ satisfies decay property (M3), hence, has finite first moment. For the second integral, we use that $|z \Phi(z) | \leq C|z|^{-l+1}$ and $2-q < l-d-1$ to obtain
		\[ 
			\int_{|z| > |x|/(2 \epsilon)}  |z|\, \Phi(z)\, dz \leq  C\int_{ |x|/(2 \epsilon)}^{+\infty} r^{-l+1} r^{d-1} dr = C \left( \frac{|x|}{2 \epsilon} \right)^{d-l+1} \leq C \epsilon |x|^{q-2}.
		\]
Therefore, \eqref{regularizedincreasing1} follows and we conclude that both $K(x)$ and $K_\epsilon(x)$ are strictly increasing in each coordinate for $|x| >2$.

\smallskip

We now show that $K(x)$ and $K_\epsilon(x)$ become arbitrarily large as $|x| \to +\infty$. This is immediate for $K(x)$, since $q>0$. Furthermore, since $K(x)$  and $\Phi_\epsilon(x)$ are nonnegative and $\int_{\Rd} \Phi_\epsilon(x)\,dx = 1$, we also have
\[ K_\epsilon(x)  \geq \int_{|y|< |x|/2} K(x-y) \Phi_\epsilon(y) dy \geq \min_{|z| \geq |x|/2} K(z) \xrightarrow{|x| \to +\infty} +\infty. \]

Thus, for both $\E$ and $\E_\epsilon$, there exists a minimizer in $\P(\Rd)$ \cite[Theorem 3.1]{SiSlTo2014}. This minimizer is compactly supported when $\eps>0$ is sufficiently small, thus it belongs to $\P_2(\Rd)$  \cite[Theorem 1.4]{CCP}.
\end{proof}

\medskip

\begin{remark}\label{rem:inf_equality} 
{\rm
Since $\P_2(\Rd)\subset\P(\Rd)$, the infimum of the energy $\E$ or $\E_\epsilon$ over $\P(\Rd)$ is less than or equal to the infimum over $\P_2(\Rd)$. Due to the fact that minimizers have compact support \cite[Lemma 2.10]{CCP}, the converse holds, as well. Hence,
	\[
		\inf_{\mu\in\P(\Rd)}\E(\mu) = \inf_{\mu\in\P_2(\Rd)}\E(\mu) \quad\text{ and }\quad \inf_{\mu\in\P(\Rd)}\E_\epsilon(\mu) = \inf_{\mu\in\P_2(\Rd)}\E_\epsilon(\mu)
	\]
for all $\eps>0$ sufficiently small.
}
\end{remark}

\bigskip

Now we prove the regularized energies $\E_\eps$ converge to $\E$ in the sense of $\Gamma$-convergence.

\medskip

\bthm[$\Gamma$-convergence of regularized energies]\label{thm:gammaconv} Suppose the energy $\E$ satisfies (E1)--(E2) and the mollifier $\varphi$ satisfies (M1)--(M3). Then the sequence of regularized energies $\{\E_\epsilon\}_{\epsilon>0}$  $\Gamma$-converges to the energy with respect to $(\P_2(\Rd),d_W)$. That is,
	\begin{itemize}
		\item[(i)] \emph{(Lower semicontinuity)} For any $\{\mu_\eps\}_{\eps>0}\subset\P_2(\Rd)$ and $\mu \in \P_2(\Rd)$ such that it holds $\lim_{\eps\arrow0}d_W(\mu_\epsilon,\mu)=0$, we have that
	\[
		\liminf_{\eps\arrow 0} \E_\eps(\mu_\eps) \geq \E(\mu).
	\]
		\item[(ii)] \emph{(Recovery sequence)} For any $\mu\in\P_2(\Rd)$ there exists  $\{\nu_\eps\}_{\eps>0}\subset\P_2(\Rd)$ such that
	\[
		\lim_{\eps\arrow 0}d_W(\nu_\epsilon,\mu)=0 \text{ and }\lim_{\epsilon\arrow 0} \E_\epsilon(\nu_\epsilon) = \E(\mu).
	\]
	\end{itemize}
\ethm

\begin{proof} We will prove this theorem in two steps.

\medskip

\noindent{Step 1} (Lower semicontinuity). By Lemma \ref{moveepsilontomu}, we have $\E_\eps(\mu_\eps) =  \E (\mu_\eps* \varphi_\eps)$. Furthermore,
	\beqn\label{eqn:lsc_middle_step}
		d_W(\mu,\mu_\eps*\varphi_\eps) \leq d_W(\mu,\mu_\eps) + d_W(\mu_\eps,\mu_\eps*\varphi_\eps),
	\eeqn
where the first term approaches zero by hypothesis and the second term approaches zero by Lemma~\ref{lem:AGS_appr_conv}. 

As the interaction potential $K$ is lower semicontinuous and bounded from below,  the Portmanteau Theorem \cite[Theorem 1.3.4]{van1996weak} ensures the energy $\E$ is lower semicontinuous with respect to convergence in $\P_2(\Rd)$. Indeed, the Portmanteau Theorem states that the weak-* convergence of $\mu_\eps \times \mu_\eps$ to $\mu \times \mu$ is equivalent to the fact that
	\[
		\liminf_{\eps\arrow 0} \iint_{\R^n\times\R^n} f(x-y)\,d(\mu_\eps \times\mu_\eps)(x,y) \geq \iint_{\R^n\times\R^n} f(x-y)\,d(\mu \times \mu)(x,y)
	\] 
when $f$ is lower-semicontinuous and bounded from below. Hence, we obtain
	\[
	 \liminf_{\eps\arrow 0} \E_\eps (\mu_\eps) = \liminf_{\eps\arrow 0} \E(\mu_\eps*\varphi_\eps) \geq \E(\mu). 
	\] 

\noindent{Step 2} (Recovery sequence). Let $\mu\in\P_2(\Rd)$ be arbitrary. We define a recovery sequence $\nu_\epsilon$ for the measure $\mu$ by using the heat kernel. Let $\psi(x)=(4\pi)^{-d/2}\,e^{-|x|^2/4}$, and define $\psi_{\delta(\epsilon)}(x) := \delta(\epsilon)^{-d} \psi(x/\delta(\epsilon))$ where 
\begin{align} \label{deltaeqn}
\delta(\epsilon) := \epsilon^{1/2d}.
\end{align}
Clearly the function $\psi$ satisfies the assumptions (M1)--(M3) in Subsection \ref{regenergysec}; hence, it is an admissible mollifier. Define the measure
	\[
	\nu_{\epsilon}:= \mu * \psi_{\delta(\epsilon)}.
	\]
Then, by Lemma \ref{lem:AGS_appr_conv}, $\nu_{\eps}\in\P_2(\Rd)$ for all $\eps>0$, and $d_W(\nu_{\eps},\mu)\arrow 0$ as $\eps\arrow 0$.

We now show that $\lim_{\eps\arrow 0}\E_{\eps}(\nu_\eps)=\E(\mu)$. By Lemma \ref{moveepsilontomu},
	\beqn
		\begin{aligned}
			|\E_\eps(\nu_\eps)-\E(\mu)| &= |\E(\mu*\psi_{\delta(\eps)}*\varphi_\eps)-\E(\mu)| = |\E_{\delta(\eps)}(\mu*\varphi_\eps)-\E(\mu)| \\
																																		&\leq |\E_{\delta(\eps)}(\mu*\varphi_\eps)-\E_{\delta(\eps)}(\mu)|+|\E_{\delta(\eps)}(\mu)-\E(\mu)|
																	=: I_1 + I_2.
		\end{aligned}
		\label{eqn:recseq1}
	\eeqn
	It suffices to show that $I_1, I_2 \to 0$ as $\epsilon \to 0$.
We estimate $I_2$ first. As $\E_{\delta(\eps)}(\mu)$ is regularized using the heat kernel $\psi_{\delta(\eps)}$, Proposition \ref{prop:monoton} ensures that $\Er_{\delta(\eps)}(\mu)\leq \Er(\mu)$ for all $\eps>0$. Hence,
	\[
		\limsup_{\eps\arrow 0} \E_{\delta(\eps)}(\mu) \leq \Er(\mu) + \limsup_{\eps\arrow 0} \Ea_{\delta(\eps)}(\mu).
	\]
Again, by Lemma \ref{moveepsilontomu}, $\Ea_{\delta(\eps)}(\mu)=\Ea(\mu*\psi_{\delta(\eps)})$. Since $d_W(\mu*\psi_{\delta(\eps)},\mu)\arrow 0$ as $\eps\arrow 0$ and the attractive interaction kernel $K^a$ is a continuous function with at most quadratic growth,
	\[
		\limsup_{\eps\arrow 0} \Ea_{\delta(\eps)} (\mu) = \limsup_{\eps\arrow 0} \Ea(\mu*\psi_{\delta(\eps)}) = \Ea(\mu).
	\]
Therefore
	\[
	 \limsup_{\eps\arrow 0} \E_{\delta(\eps)}(\mu) \leq \E(\mu).
	\]
On the other hand, by the lower semicontinuity of the energy $\E$, as proved in the first step of the proof, we get that
	\[
		\liminf_{\eps\arrow 0} \E_{\delta(\eps)}(\mu) = \liminf_{\eps\arrow 0} \E(\mu*\psi_{\delta(\eps)}) \geq \E(\mu);
	\]
hence, $I_2=|\E_{\delta(\eps)}(\mu)-\E(\mu)|\arrow 0$ as $\eps\arrow 0$.

\medskip

To estimate $I_1$, recall that by Remark \ref{rem:lambdaconv} and \ref{rem:differentiability} the regularized energies $\pm \E_{\delta(\eps)}$ are $\lambda_{\delta(\epsilon)}$-convex with $\lambda_{\delta(\epsilon)}=-C_\varphi \delta(\epsilon)^{-d}$ and have metric local slope $\|\nabla K * \psi_{\delta(\epsilon)} * \mu \|_{L^2(\mu)}$. Therefore, they satisfy the following HWI-type inequality \cite[Theorem 2.4.9]{AGS},
	\beqn
	  |\E_{\delta(\eps)}(\mu*\varphi_\eps)-\E_{\delta(\eps)}(\mu)| \leq 2 \|\nabla K*\psi_{\delta(\eps)}*\mu\|_{L^2(\mu)}d_W(\mu*\varphi_\eps,\mu) - \frac{\lambda_{\delta(\epsilon)}}{2}d_W^2(\mu*\varphi_\eps,\mu). 
	  \nonumber 
	\eeqn
Consequently, for $\eps>0$ sufficiently small, so that $d_W(\mu*\varphi_\eps,\mu)\leq 1$, we may use the bound $\|\nabla K * \psi_{\delta(\epsilon)} * \mu \|_{L^2(\mu)} \leq C_{\psi, \mu}\delta(\epsilon)^{1-d}$ from Remark \ref{rem:differentiability} and the definition of $\delta(\epsilon)$ from equation (\ref{deltaeqn}) to obtain
	\[
	I_1 =  |\E_{\delta(\eps)}(\mu*\varphi_\eps)-\E_{\delta(\eps)}(\mu)|   \leq C(\delta(\epsilon)^{1-d} +\delta(\epsilon)^{-d})\, d_W(\mu*\varphi_\epsilon, \mu) \leq C \epsilon\, \delta(\epsilon)^{-d} = C \epsilon^{1/2}
	\]
Therefore $I_1\arrow 0$ as $\eps\arrow 0$, which completes the proof. 
\end{proof}

\medskip

\begin{remark}\label{rem:lsc_weak}
{\rm
Since the energy $\E$ is lower semicontinuous with respect to weak-* convergence of measures in $\P(\Rd)$ and convergence in $\P_2(\Rd)$ implies weak-* convergence in $\P(\Rd)$, the conclusion of Theorem \ref{thm:gammaconv}(i) is also true with respect to weak-* convergence in $\P(\Rd)$.
}
\end{remark}

\medskip

As a result of the strong confining forces induced by the attractive part of the kernel $K$, we obtain the following compactness result for the sequence of energies $\E_\eps$ in $\P(\Rd)$.

\medskip

\bprop[Compactness in $\P(\Rd)$]\label{prop:compactness}
Suppose the energy $\E$ satisfies (E1)--(E2) and the mollifier $\varphi$ satisfies (M1)--(M3). Let $\{\mu_\eps\}_{\eps>0}\subset\P(\Rd)$ be a sequence such that, for all $\eps>0$ sufficiently small, $\E_\eps(\mu_\eps) \leq C$ for some constant $C>0$. Then $\{\mu_\eps\}_{\eps>0}$ has a  subsequence which is, up to translations, convergent with respect to the weak-* topology in $\P(\Rd)$.
\eprop

\begin{proof}
We will use Lemma \ref{lem:conc_comp} and argue by contradiction, as in \cite{SiSlTo2014}. Note that as in the proof of Proposition \ref{prop:exist}, $K(x)$ and $K_\eps(x)\arrow+\infty$ as $|x|\arrow +\infty$ for $\eps>0$ sufficiently small. This ensures that the energy $\E_\eps$ exhibits long-range confinement. We will use this fact to eliminate ``vanishing'' and ``dichotomy'' possibilities of the sequence $\{\mu_\eps\}_{\eps>0}$.

Suppose a subsequence of $\{\mu_\eps\}_{\eps>0}$, which we still index by $\eps>0$, ``vanishes'' in the sense of Lemma \ref{lem:conc_comp}(ii). Then for any $\delta>0$ and $R>0$, there exists $\eps_0>0$ such that for all $\eps<\eps_0$ and $x\in\Rd$, we have
	\[
		\mu_\eps(\Rd\setminus B_R(x))\geq 1-\delta.
	\]
Hence, for all $\eps<\eps_0$,
	\beqn\label{vanishingmiddlestep}
		\iint_{|x-y|\geq R} d\mu_{\eps}(x)d\mu_{\eps}(y) = \int_{\Rd}\!\left(\int_{\Rd\setminus B_R(x)}d\mu_\eps(y)\right)d\mu_\eps(x) \geq 1-\delta.
	\eeqn
Since $K_\eps(x)\arrow+\infty$ as $|x|\arrow\infty$, given a constant $B>0$, there exists $\tilde{R}>0$ such that for all $|x|\geq \tilde{R}$, we have $K_\eps(x)\geq B$. Let $\delta<1/2$, and choose $\eps_0$ depending on $\delta$ and $\tilde{R}$ such that \eqref{vanishingmiddlestep} holds. Then, using \eqref{vanishingmiddlestep} and the fact that $K_\eps\geq 0$ for all $\eps>0$, we get that
	\beqn
		\begin{aligned}
			\E_\eps(\mu_\eps) &= \iint_{|x-y|<\tilde{R}} K_{\eps}(x-y)\,d\mu_{\eps}(x)d\mu_{\eps}(y) + \iint_{|x-y|\geq\tilde{R}} K_{\eps}(x-y)\,d\mu_{\eps}(x)d\mu_{\eps}(y) \\
												&\geq \iint_{|x-y|\geq\tilde{R}} K_{\eps}(x-y)\,d\mu_{\eps}(x)d\mu_{\eps}(y) \geq (1-\delta)B.
		\end{aligned}
		\nonumber
	\eeqn
This  contradicts the uniform bound $\E_\eps(\mu_\eps)\leq C$ when $B$ is sufficiently large; hence, ``vanishing'' does not occur.

Suppose that for a subsequence of $\{\mu_\eps\}_{\eps>0}$, which again we index by $\eps>0$, ``dichotomy'' occurs. Then, since $K_\eps\geq 0$ for all $\eps>0$,
	\beqn
		\begin{aligned}
			\liminf_{\eps\arrow 0} \E_\eps(\mu_\eps) &\geq \liminf_{\eps\arrow 0} \int_{\Rd\setminus B_R^\pr(x_\eps)} \! \int_{B_R(x_\eps)} K_\eps(x-y)\,d\mu_\eps^1(x)d\mu_\eps^2(y) \\
			&\geq \inf_{|x|\geq R^\pr-R} K_{\eps}(x)\alpha(1-\alpha),
		\end{aligned}
		\nonumber
	\eeqn
where $R>0$, the sequence $\{x_\eps\}_{\eps>0}$, and the measures $\mu_\eps^1$ and $\mu_\eps^2$ are defined as in Lemma \ref{lem:conc_comp}(iii), and $R^\pr>R$ is arbitrary. Thus, again using the growth of $K_\eps$, we get that 
	\[
	\liminf_{\eps\arrow 0} \E_\eps(\mu_\eps)\geq +\infty,
	\]
contradicting the uniform bound on $\E_\eps(\mu_\eps)$.

Therefore, Lemma \ref{lem:conc_comp} ensures the sequence $\{\mu_\eps\}_{\eps>0}$ is tight up to a translation by a sequence $\{y_\eps\}_{\eps>0}\subset\Rd$. However, the energies $\E_\eps$ are translation invariant, so we can replace $\mu_\eps$ by $\mu_\eps(\cdot+y_\eps)$ which yields a tight sequence. Therefore by Prokhorov's theorem (cf. \cite[Theorem 1.3.9]{van1996weak}), there exists a subsequence $\{\mu_\eps\}_{\eps>0}$, still indexed by $\eps>0$, and a measure $\mu\in\P(\Rd)$ such that, up to translations,
$\mu_\eps \arrow \mu$ with respect to weak-* convergence in $\P(\Rd)$ as $\eps\arrow 0$.
\end{proof}

\bigskip

Classically, if a sequence of $\Gamma$-convergent functionals also satisfies a compactness property, then, up to a subsequence, minimizers converge to a minimizer of the limiting functional. Though our compactness result, Proposition \ref{prop:compactness}, is established in a weaker space ($\P(\Rd)$) than the topology in which the energies $\Gamma$-converge ($\P_2(\Rd)$), we still obtain the following corollary on the convergence of minimizers.

\medskip

\begin{corollary}[Convergence of minimizers]\label{cor:conv_min}Suppose the energy $\E$ satisfies (E1)--(E2) and the mollifier $\varphi$ satisfies (M1)--(M3). 
For any $\eps>0$ sufficiently small let $\mu_\eps\in\P_2(\Rd)$ be a minimizer of the energy $\E_\eps$. Then there exists $\mu\in\P_2(\Rd)$ such that, up to a subsequence and translations, $\mu_\eps\arrow\mu$ in $\P_2(\Rd)$ as $\eps\arrow 0$, and $\mu$ minimizes the energy $\E$ over $\P_2(\Rd)$.
\end{corollary}

\begin{proof}
Since the sequence $\{\mu_\eps\}_{\eps>0}\subset\P_2(\Rd)$ consists of minimizers of $\E_\eps$, for $\eps>0$ sufficiently small, there exists a constant $C>0$ so that $\E_\eps(\mu_\eps)\leq C$. Then by Proposition \ref{prop:compactness}, there exists a measure $\mu\in\P(\Rd)$ such that a subsequence of $\{\mu_\eps\}_{\eps>0}$, which we still denote by $\mu_\eps$, weak-* converges to $\mu$ in $\P(\Rd)$.

To show that $\mu$ minimizes $\E$ over $\P_2(\Rd)$ we proceed in two steps. First, consider an arbitrary measure $\nu\in\P_2(\Rd)$. By Theorem \ref{thm:gammaconv} (ii), there exists a sequence $\{\nu_\eps\}_{\eps>0}\subset\P_2(\Rd)$ so that $d_W(\nu_\eps,\nu)\arrow 0$ as $\eps\arrow 0$ and
	\[
		\lim_{\eps\arrow 0} \E(\nu_\eps)=\E(\nu).
	\]
By Theorem \ref{thm:gammaconv} (i), Remark \ref{rem:lsc_weak}, and the fact that the measures $\mu_\eps$ are minimizers of $\E_\eps$ over $\P_2(\Rd)\subset\P(\Rd)$, we obtain
	\beqn\label{midstep}
		\E(\mu) \leq \liminf_{\eps\arrow 0} \E_\eps(\mu_\eps) \leq \liminf_{\eps\arrow 0} \E_\eps(\nu_\eps) =\E(\nu).
	\eeqn
Thus, $\mu\in\P(\Rd)$ has less energy than any other measure $\nu\in\P_2(\Rd)$, so by Remark \ref{rem:inf_equality}, $\mu$ minimizes $\E$ over $\P(\Rd)$, as well. Consequently, \cite[Lemma 2.10]{CCP} ensures $\mu$ is compactly supported; hence it is in $\P_2(\Rd)$.

Finally, we show that in fact $\mu_\eps$ converges to $\mu$ in $\P_2(\Rd)$. In \cite[Lemma 2.10]{CCP} the authors show that
	\[
		\text{diam}(\supp\mu_\eps) \leq M_\eps
	\]
where $M_\eps:=\sqrt{d}(4r_\eps+(\left\lceil 1/m_\eps\right\rceil-1)(4r_\eps+2R_\eps))$. Here
	\[
		m_\eps:=\frac{C-\E_\eps(\mu_\eps)}{C-K_\eps^{\min}}
	\]
with $C>0$ so that $\E_\eps(\mu_\eps)\leq C$ for $\eps>0$ sufficiently small. 
$K_\eps^{\min}$ denotes the absolute minimum value of the interaction kernel $K_\eps$, $r_\eps$ is chosen such that $K_\eps(x)\geq C$ for all $|x|\geq r_\eps$, and $R_\eps$ is the radius after which the kernel $K_\eps$ is strictly increasing. We denote the corresponding quantities for the unregularized energy $\E$ by removing the subscript $\eps$.

We show that $M_\eps$, the upper bound on the diameter of the support of $\mu_\eps$, is bounded by some $\tilde{M}>0$ for all $\eps>0$ sufficiently small. 
We have $\limsup_{\epsilon \to 0} -\E_\eps(\mu_\eps) \leq -\E(\mu)$ and $K_\eps\arrow K$ uniformly on compact sets away from the origin. Thus, $m_\eps\arrow m$ as $\eps\arrow 0$, so $\left\lceil 1/m_\eps\right\rceil \leq \left\lceil 1/m\right\rceil+1 $ for $\eps >0$ sufficiently small. 

As shown in the proof of Proposition \ref{prop:exist}, $K_\eps$ is strictly increasing for $|x|\geq 2$, so we may take $R_\eps=2$. Likewise, by estimate (\ref{regularizedincreasing1}), there exists $r>2$ so that $K_{\eps}(x) \geq C$ for $|x|>r$. Therefore
	\[
		M_\eps \leq \tilde{M} :=\sqrt{d}(4r+(\left\lceil 1/m\right\rceil)(4r+4))
	\]
for all $\eps>0$ sufficiently small. 

This shows that the diameter of the support of $\mu_\epsilon$ is uniformly bounded by $\tilde{M}$. Consequently, the second moments of the sequence $\mu_\epsilon$ are uniformly integrable
\[ \lim_{k \to +\infty} \int_{\{|x|^2 \geq k\}} |x|^2 \, d \mu_\epsilon  =  \lim_{k \to +\infty} \int_{\{\sqrt{k} \leq |x| \leq \tilde{M}\} } |x|^2 \, d \mu_\epsilon = 0 \]
Since $\mu_\epsilon$ converges to $\mu$ with respect to the weak-* topology on $\P(\Rd)$ and $\mu_\epsilon$ has uniformly integrable second moments,  $\mu_\eps\arrow\mu$ in $\P_2(\Rd)$ with respect to  $d_W$ \cite[Proposition 7.1.5]{AGS}.
\end{proof}

\medskip

\begin{remark}[Negative attraction power]\label{rem:negative_attr_power}
{\rm
The condition $q>0$ is a sufficient condition for Propositions \ref{prop:exist} and \ref{prop:compactness}. If $q<0$, both $K$ and $K_\eps$ converge to zero as $|x|\arrow\infty$. In \cite{SiSlTo2014}, the authors characterized the existence of minimizers for these types of potentials using the existence of a measure for which the energy is negative. This is clearly true for $\E$ and $\E_\eps$, as one can consider the characteristic function of a sufficiently large ball $B_R(0)$. However, we still require $q>0$ for the compactness result in Proposition \ref{prop:compactness}, and we believe that it is a necessary condition, as well.
}
\end{remark}

\section{More general interaction energies}\label{sec:generalization}

Our results of sections \ref{sec:prelim} and \ref{sec:conv} can be extended to more general nonlocal interaction energies 
	\beqn\label{eqn:energy_repeat}
		\E(\mu) = \int_{\Rd}\!\int_{\Rd} K(x-y)\,d\mu(x) d\mu(y)
	\eeqn 
that are not necessarily repulsive-attractive and are defined via an interaction kernel $K:\Rd\to\R\cup\{+\infty\}$ that satisfies the following hypotheses for some fixed $R>1$:
\medskip
		\begin{itemize}	
				\addtolength{\itemsep}{6pt}
						\item[(H1)] $K$ is even, i.e., $K(x) = K(-x)$ for all $x \in \Rd$.

						\item[(H2)] $K \in L^1_\loc(\Rd) \cap C^1(\Rd \setminus B_R(0))$ and $| \nabla K(x)| \leq C(1+|x|)$ for $|x| >R$.
						\item[(H3)] For $|x|>R$, $K$ is strictly increasing in each coordinate  and
							\[
								\lim_{|x|\to +\infty} K(x) = + \infty.
							\] 
					
						\item[(H4)] There exists a function $K^a\in C(\Rd)$ so that $K^r:=K-K^a$ is superharmonic. 
						\item[(H5)] For $|x|$ large, $K^r$ and $K^a$ have at most quadratic growth, i.e.,
							\beqn
								 |K^a(x)| + |K^r(x)|  \leq C(1+|x|^2) \text{ for } |x| >R. \nonumber
							\eeqn
		\end{itemize}
\medskip

We briefly recall the definition of superharmonicity and some immediate properties of superharmonic functions.
\begin{defi}
A function $f\in L^1_{\text{\rm loc}}(\Rd)$ is superharmonic if
	\[
		f(x) \geq \frac{1}{|B_r(x)|} \int_{B_r(x)} f(y)\,dy
	\]
for a.e. $x\in\Rd$ and every $r>0$, where $|B_r(x)|$ denotes the volume of the ball.
\end{defi}

\medskip

\begin{remark}[Properties of superharmonic functions]{\rm
If $f \in L^1_\loc(\Rd)$ is superharmonic, then $f$ is lower semicontinuous, it is bounded below on compact sets,  and its distributional Laplacian satisfies $\Delta f \leq 0$ \cite[Theorem 9.3]{LiLo}.}
\end{remark}

\medskip

Interaction energies \eqref{eqn:energy_repeat} defined via kernels satisfying assumptions (H1)--(H5) cover a wide range of applications. In particular, in any dimension $d \geq 1$, they include energies with $K$ of one of the following forms:
\smallskip
\begin{itemize}\addtolength{\itemsep}{6pt}
\item Attractive-repulsive power-law kernels \eqref{repulsive attractive power law}, for $2-d \leq p< 0 < q \leq 2$. (This includes Newtonian repulsion, $p = 2-d$, in $d\geq 3$.)
\item Nonsingular attractive-repulsive power-law kernels \eqref{repulsive attractive power law}, for $0 < p < q \leq 2$. (This includes Newtonian repulsion, $p=1$, in $d=1$.)
\item For $d \geq 2$, attractive-repulsive kernels with logarithmic repulsion, $K(x) = |x|^q/q - \log(|x|)$ for $0< q \leq 2$. (This includes Newtonian repulsion in $d=2$.)
\item Attractive-repulsive kernels with power-law behavior at the origin and infinity, i.e., radial kernels for which there exist $0<R_1<R_2$ and $C_1$, $C_2>0$ so that for either $2-d \leq p<0 <q\leq 2$ or $0 < p< q \leq 2$,
\smallskip
	\begin{itemize}\addtolength{\itemsep}{6pt}
		 \item[] $K(x)=C_1\,|x|^p$ if $d\geq 1$ or $K(x) = -C_1 \log(|x|)$ if $d\geq 2$ on $\overline{B_{R_1}(0)}$;
		 \item[] $K(x)$ is continuous on $\overline{B_{R_2}(0)} \setminus B_{R_1}(0)$; and,
		 \item[] $K(x)=C_2\,|x|^q$  on  $\Rd\setminus B_{R_2}(0)$.
	\end{itemize}
 \end{itemize}

\medskip

\begin{remark}\label{rem:bddbelow_lsc}
{\rm
Any function $K$ satisfying the hypotheses (H1)--(H5) is locally integrable, bounded from below, and lower semicontinuous. Consequently, without loss of generality, we will assume that $K$ is nonnegative, since the minimizers of an interaction energy $\E$ with kernel $K$ are the same as the minimizers of $\E$ with kernel $K - \inf K$.
}
\end{remark}

\medskip

Our  results continue to hold for $K$ satisfying (H1)--(H5), provided that we define the regularized energies $\E_\eps$ via compactly supported mollifiers. That is, we replace mollifier assumption (M3) by
\medskip
		\begin{itemize}
				\addtolength{\itemsep}{6pt}
						\item[(M3')] $\varphi$ is compactly supported in $\Rd$.
		\end{itemize}
\medskip
This condition  eliminates possible oscillations in $K_\eps$ for large values of $|x|$, allowing us to establish growth and monotonicity properties of the regularized kernels when $|x|$ is sufficiently large.

We now  describe how our previous results naturally extend to $K$ satisfying (H1)--(H5). The key result in section \ref{sec:prelim} is Proposition \ref{prop:monoton}. Assumption (H4) ensures that the potentially singular part of the kernel, $K^r$, is locally integrable and superharmonic. Hence, Proposition \ref{prop:monoton} continues to hold since, again taking $\psi_\epsilon$ to be the heat kernel,
\[ \frac{\partial}{\partial \epsilon} K^r_\epsilon= 2\psi_\epsilon * K^r * (\partial_\epsilon \psi_\epsilon ) = 2\psi_\epsilon* K^r * (\Delta \psi_\epsilon) =2 \psi_\epsilon * (\Delta K^r) * \psi_\epsilon \leq 0 , \]
by the superharmonicity of $K^r$.

The remaining results of section \ref{sec:prelim} also continue to hold for energies satisfying (H1)--(H5) and mollifiers satisfying (M1), (M2), and (M3'). In particular, (H2) ensures $K \in L^1_\loc(\Rd)$ so that Remark \ref{rem:lambdaconv} holds with $\lambda_\epsilon = C_\varphi \epsilon^{-d-2}$, since for $0< \epsilon<1$,
\begin{align} \label{newlambdaepsilon}
\|D^2 K_\epsilon\|_{L^\infty} \leq \|D^2 \varphi_\epsilon \|_{L^\infty} \int_{\supp \varphi} K(x)\,dx \leq C_\varphi \,\epsilon^{-d-2} < +\infty.
\end{align}
 Likewise, assumption (H2) ensures that Remark \ref{rem:differentiability} holds. As before, we may choose a cutoff function $\eta$ to separate $K$ into a locally integrable component $K^s$ near the origin and a continuously differentiable component $K^n$ away from the origin. However, since for $0< \epsilon<1$, 
 \begin{align} \label{newmetricslope} \|\nabla K^s_\epsilon \|_{L^\infty} \leq \|\nabla \varphi_\epsilon\|_{L^\infty} \int_{\supp \varphi}K^s(x)\,dx \leq C_\varphi\, \epsilon^{-1-d} , 
 \end{align} we merely obtain the bound $2 \| \nabla K_\epsilon *\mu \|_{L^2(\mu)} \leq C_{\varphi,\mu} \,\epsilon^{-1-d}$ on the metric local slope.

Now, we describe how to extend the results from section \ref{sec:conv}. Proposition \ref{prop:exist} merely requires that $K$ and $K_\epsilon$ are locally integrable, lower semicontinuous, strictly increasing in each coordinate outside of some ball, and become arbitrarily large as $|x| \to +\infty$. To see that the last two properties hold for $K_\epsilon$, note that for $0< \epsilon<1$,
\begin{gather} 
	K_\epsilon(x) - K_\epsilon(y) = \int_{\supp \varphi} [K(x-z) - K(y-z)] \varphi_\epsilon(z)\,dz, \text{ and} \nonumber \\
	 K_\epsilon(x) \geq \min_{z \in \supp \varphi} K(x-z). \nonumber
\end{gather}
Consequently the properties for $K_\epsilon$ are immediate consequences of the corresponding properties for $K$, which hold by (H3).

The proof of Theorem \ref{thm:gammaconv}  is nearly identical for our more general energies, though we must take $\delta(\epsilon) =\epsilon^{1/2(d+2)}$ to compensate for the inferior bounds on the convexity constant $\lambda_\epsilon$ (\ref{newlambdaepsilon}) and the metric slope (\ref{newmetricslope}). Finally, the proofs of Proposition \ref{prop:compactness} and Corollary \ref{cor:conv_min} again use that $K$ and $K_\epsilon$ become arbitrarily large (uniformly in $\epsilon$) as $|x| \to +\infty$, are strictly increasing in each coordinate outside of some ball (uniformly in $\epsilon$), and are both nonnegative.

Consequently, though we choose to state our main results in sections \ref{sec:prelim} and \ref{sec:conv} for energy functionals with repulsive-attractive power-law potentials satisfying (E1)--(E2) (see Remarks \ref{rem:powerlaw_generalization} and \ref{rem:powerlawchoices}), these results naturally extend to a wider class of energy functionals \eqref{eqn:energy_repeat} with $K$ satisfying (H1)--(H5).

\medskip

\begin{remark}[More singular kernels]\label{rem:moresingular}
{\rm
The main difficulty in extending our results to interaction kernels that are more singular than the Newtonian potential at the origin (i.e., $-d<p<2-d$) lies in the construction of a recovery sequence in the proof of Theorem~\ref{thm:gammaconv}. Although one could obtain a monotonicity result as in Proposition~\ref{prop:monoton} for such kernels by regularizing them via fractional heat kernels, unfortunately, to our knowledge, in the parameter regime $-d<p<2-d$, the decay estimates on fractional heat kernels do not guarantee that they have finite second moments. Therefore they are not admissible candidates in constructing our recovery sequences.
}
\end{remark}

\section{Convergence of gradient flows} \label{gradflowsection}

In this section, we apply our result on the $\Gamma$-convergence of regularized energies satisfying (E1)--(E2) to show that if the Wasserstein gradient flows of $\E_\epsilon$ are bounded in $L^\infty(\Rd)$, they converge to a metric space generalization of the gradient flow of $\E$, known as a \emph{curve of maximal slope}. 

As previously mentioned, without regularization, energies $\E$ satisfying (E1)--(E2) are not convex (or $\lambda$-convex for $\lambda <0$). Consequently, they fall outside the scope of much of the existing theory on well-posedness of Wasserstein gradient flows \cite{AGS, 5person}. However, the regularized interaction energies $\E_\epsilon$ are $\lambda_\epsilon = C_\varphi \epsilon^{-d}$ convex, so that for fixed $\epsilon >0$, their Wasserstein gradient flows exist and are unique \cite{5person}. Thus, by showing these gradient flows $\Gamma$-converge to the curve of maximal slope for the unregularized energy, we provide a link between the well-understood case of convex gradient flow and emerging results on the Wasserstein gradient flow of non-convex energies \cite{CarrilloLisiniMainini, CraigOsgood, AmbrosioSerfaty, Mainini, CarrilloRosado, BertozziLaurentRosado, MaininiGinzburg, CarrilloMcCannVillani}.

We restrict our attention to the space of probability measures with bounded density
 \begin{align} \label{bddprobmeas}
  \P_{2,R}(\Rd) := \{ \mu \in \P_{2,ac}(\Rd) \colon \|\mu\|_{L^\infty(\Rd)} \leq R \},
  \end{align}
for any $R>0$, due to the fact that, though $\E$ is not convex (or $\lambda$-convex for $\lambda <0$), once it is restricted to $\P_{2,R}(\Rd)$, it possesses a generalized notion of convexity known as \emph{$\omega$-convexity}, where $\omega(x)$ is a log-Lipschitz modulus of convexity.

\medskip

\begin{defi}[cf.{\cite{CarrilloLisiniMainini,CraigOsgood}}]
Consider the modulus of convexity
\begin{align} \label{omegadef}
 \omega(x) &:= \begin{cases} x |\log x| & \text{ if } 0 \leq x \leq e^{-1-\sqrt{2}} ,\\ \sqrt{x^2+ 2(1+\sqrt{2})e^{-1-\sqrt{2}}x} & \text{ if }x > e^{-1-\sqrt{2}} .  \end{cases}  
\end{align}
Then $\E:\P_{2,R}(\Rd) \to \R \cup \{+\infty\}$ is \emph{$\omega$-convex} if, for all constant speed geodesics $\mu_\alpha:[0,1] \to \P_{2,R}(\Rd)$, there exists $C>0$ so that
\[ \E(\mu_\alpha) \leq (1-\alpha) \E(\mu_0) + \alpha \E (\mu_1) +  \frac{C}{2} \left[ (1-\alpha) \omega(\alpha^2 d_W^2(\mu_0,\mu_1)) + \alpha \omega( (1-\alpha)^2 d_W^2(\mu_0,\mu_1)) \right].
\]
\end{defi}

\medskip

This notion has been used, explicitly and implicitly, in many previous works on non-convex gradient flow \cite{ CarrilloLisiniMainini, CarrilloMcCannVillani}. For our purposes, $\omega$-convexity allows us to define a notion of upper gradient for $\E$, which is an essential component in defining its curve of maximal slope.
 
Both from the perspective of energy minimization and the dynamics of the aggregation equation (\ref{agg eqn}), the restriction to bounded densities is quite natural when $\E$ is a repulsive-attractive interaction energy satisfying (E1)--(E2) with $p=2-d$. Indeed, several results in the literature support the fact that both the energy minimization and evolution take place in the space of bounded functions. For quadratic attraction (i.e., for $q=2$), the global minimizer of $\E$ is the characteristic function on a ball \cite{CDM14,ChFeTo14}. Likewise, from the dynamics point of view, if  the initial data belongs to $\P_2(\Rd) \cap L^\infty(\Rd)$ and has compact support, then it remains bounded for all time and converges to the characteristic function on the unit ball \cite{BertozziLaurentLeger}. For more general attraction powers ($0<q\leq 2$), all compactly supported local minimizers are  bounded \cite{CDM14} and smooth, compactly supported classical solutions of the aggregation equation remain bounded for all time \cite[Lemma 1]{BalagueCarrilloYao}.

For $p \neq 2-d$, the restriction to measures with bounded density is perhaps less natural, since the minimizers may concentrate mass on sets of measure zero and classical solutions to the aggregation equation can approach these steady states asymptotically \cite{BalagueCarrilloLaurentRaoul}. However, since our proof regarding the $\Gamma$-convergence of gradient flows works for any choice of parameters $2-d \leq p <0 < q \leq 2$, we choose to include this regime for the sake completeness. 

\medskip

\begin{remark}[Gradient flow vs. aggregation equation]\label{rem:grad_flow_agg_eqn}
{\rm
While the aggregation equation (\ref{agg eqn}) does inform our choice of the space of measures with bounded density (\ref{bddprobmeas}),  the relationship between weak solutions of the aggregation equation and the Wasserstein gradient flow of the interaction energy is purely formal for the unregularized (hence nonconvex) interaction energy. There is hope this relationship can be made rigorous following the approach of Ambrosio and Serfaty \cite{AmbrosioSerfaty}, but we leave this for future work.

For the regularized energy, if $\mu(t) \in \P_{2}(\Rd)$ is a weak solution of the aggregation equation (\ref{agg eqn}), then it is also a Wasserstein gradient flow of $\E_\eps$ \cite[Corollary 11.1.8]{AGS}. Below, we consider the \emph{curve of maximal slope} of $\E_\eps$ on the metric space of measures with bounded density $(\P_{2,R}(\Rd),d_W)$, and if $\mu(t)  \in \P_{2}(\Rd)$ is a weak solution of the aggregation equation with $\|\mu(t)\|_\infty \leq R$, then it is a curve of maximal slope for   the weak upper gradient $g_\epsilon(\mu):=2 \|\nabla K_\epsilon * \mu\|_{L^2(\mu)}$ on $(\P_{2,R}(\Rd),d_W)$ \cite[Theorem 11.1.3]{AGS}. The reverse implication is in general false, due to the height constraint imposed by $\P_{2,R}(\Rd)$.
}
\end{remark}

\medskip

\subsection{Curves of maximal slope and $\Gamma$-convergence}
We now briefly recall the notion of curves of maximal slope on a compete metric space $(\S, d)$. We refer the reader to the book by Ambrosio, Gigli, and Savar\'e \cite[Chapter 1]{AGS} for further details. A curve $u(t): (a,b) \to \S$ is \emph{2-absolutely continuous} if there exists $m \in L^2(a,b)$ so that 
\begin{align} \label{absctsdef} d(u(t),u(s)) \leq \int_s^t m(r)\,dr \text{ for all } a <s \leq t < b .
\end{align}
For any 2-absolutely continuous curve, the limit
\[ |u'(t)| = \lim_{s \to t} \frac{d(u(s),u(t))}{|s-t|} \]
exists for a.e. $t \in (a,b)$. Furthermore $m(t):= |u'(t)| \in L^2(a,b)$ satisfies (\ref{absctsdef}) and for any $m \in L^2(a,b)$ satisfying (\ref{absctsdef}), we have 
\[ |u'(t)| \leq m(t) \text{ for a.e. } t \in (a,b). \]

Given a functional $\F: \S \to (-\infty, + \infty]$ that is \emph{proper}, i.e., $D(\F) = \{ u \in \S : \F(u) < +\infty \} \neq \emptyset$, its upper gradient is a generalization of the modulus of the gradient from Euclidean space. Specifically, $g: \S \to [0,+\infty]$ is a \emph{strong upper gradient} for $\F$ if for every 2-absolutely continuous curve $u(t):(a,b) \to \S$ the function $g \circ u$ is measurable and
	\begin{align} \label{stronguppergradienteqn}
		|\F(u(t)-\F(u(s))| \leq \int_s^t g(u(r))|u^\pr|(r)\,dr \text{ for all } a<s\leq t<b.
	\end{align}
One example of a strong upper gradient is given by the metric local slope
	\[
		|\partial \F|(u) := \limsup_{ v \to u} \frac{(\F(u)-\F(v))_+}{d(u,v)},
	\]
when $\F$ is a $\lambda$-convex and lower semicontinuous functional \cite[Corollary 2.4.10]{AGS}.

Finally, we recall the definition of a curve a maximal slope. A locally 2-absolutely continuous curve $u: (a,b) \to \S$ is a \emph{curve of maximal slope} for $\F$ with respect to the strong upper gradient $g$ if there exists a non-increasing function $\phi$ so that $\phi(t) = \F \circ u(t)$ for a.e. $t \in (a,b)$ and
\begin{align} \label{curvemaxslopeeqn}
 \phi'(t) \leq -\frac{1}{2} |u'|^2(t) - \frac{1}{2} g^2(u(t)) \text{ for a.e. } t \in (a,b) .
 \end{align}
For all $\eps >0$, \cite[Corollary 2.4.12]{AGS} ensures that if $\mu \in D(\E_\eps)$, then there exists a curve of maximal slope $\mu_\eps(t) \colon (0,+\infty) \to \P_{2,R}(\Rd)$ for $\E_\eps$ with respect to the strong upper gradient $g_\eps(\nu) = 2 ||\nabla K_\eps *\nu\|_{L^2(d \nu)}$ satisfying $\mu_\eps(0) = \mu$.

\medskip

With these definitions in hand, we now recall a general result of Serfaty on the $\Gamma$-convergence of gradient flows on a metric space.

\medskip

\begin{theorem}[cf.{\cite[Theorem 2]{Serfaty}}]\label{thm:Serfatygammagradflow}
	Let $\F_\epsilon$ and $\F$ be functionals defined on metric spaces $(\S_\epsilon,d_\epsilon)$ and $(\S,d)$ with strong upper gradients $g_\epsilon$ and $g$, respectively. Suppose the following criteria hold:
	\begin{enumerate}
		\item ($\Gamma$-liminf convergence) There is a notion of convergence $S$ of $u_\epsilon \in \S_\epsilon$ to $u \in \S$ so that
		\[ u_\epsilon \stackrel{S}{\rightharpoonup} u \quad\text{implies}\quad \liminf_{\epsilon \to 0} \F_\epsilon(u_\epsilon) \geq \F(u). \]
		\item (Lower bound on metric derivatives) If $u_\epsilon (t) \stackrel{S}{\rightharpoonup} u(t)$ for $t \in (0,T)$, then for $\ s \in [0,T)$,
		\[ \liminf_{\epsilon \to 0} \int_0^s |u_\epsilon'|^2_{d_\epsilon}(t)\, dt \geq \int_0^s |u'|_d^2(t)\, dt. \]
		\item (Lower bound on slopes) If $u_\epsilon \stackrel{S}{\rightharpoonup} u$, then $ \liminf_{\epsilon \to 0} g_\epsilon(u_\epsilon) \geq g(u)$.
	\end{enumerate}
If $u_\epsilon(t)$ is a curve of maximal slope on $(0,T)$ for $\F_\epsilon$ with respect to $g_\epsilon$ satisfying 
\[ u_\epsilon(t) \stackrel{S}{\rightharpoonup} u(t) \text{ for } t \in (0,T) \quad \text{and} \quad \lim_{\epsilon\to 0} \F_\epsilon(u_\epsilon(0)) = \F(u(0)), \]
then  $u(t)$ is a curve of maximal slope for $\F$ with respect to $g$ and
	\begin{align*}
\lim_{\epsilon\to 0} \F_\epsilon(u_\epsilon(t)) = \F(u(t)) \quad\text{for all }  \ t \in [0, T), \\
 g_\epsilon (u_\epsilon) \to g(u) \text{ and } |u'_\epsilon|_{d_\epsilon} \to |u'|_d \text{ in } L^2_\loc(0,T).
\end{align*}
\end{theorem}

\begin{remark}\label{rem:choice_of_top}
{\rm
Although the metric $d$ induces a natural topology on $\S$, the above result admits a notion of convergence $S$ that can be induced by a weaker topology $\sigma$ on $\S$. Indeed, $u_\eps \stackrel{S}{\rightharpoonup} u$ means that $\pi_\eps(u_\eps)\stackrel{\sigma}{\rightharpoonup}u$ for some map $\pi_\eps\colon \S_\eps \to \S$ (see \cite{Serfaty} and \cite[Remark 2.0.5]{AGS} for details).
}
\end{remark}

\medskip

\subsection{$\Gamma$-convergence of the curves of maximal slope for the regularized energies}
Serfaty's result on the $\Gamma$-convergence of curves of maximal slope provides a powerful general framework. In practice, it can be  challenging to verify conditions $(i)-(iii)$. However, for the regularized interaction energies $\E_\eps$, these conditions follow from our previous results and the following HWI-type inequality, from work by the first author, which allows us to compute the strong upper gradients.

\medskip

\begin{lemma}[{c.f. \cite{CraigOsgood}}] \label{stronggradientlemma}
Suppose $\mu,\nu \in \P_{2,R}(\Rd)$ and $\E$ is an interaction energy satisfying (E1)--(E2). 
For $\omega(x)$ as in equation (\ref{omegadef}), there exists $C_{R,d,p,q}>0$ so that
\[ |\E(\mu) - \E(\nu)| \leq 2\| \nabla K*\mu\|_{L^2(\mu)} d_W(\mu,\nu) +C_{R,d,p,q}\, \omega \big(d^2_W(\mu,\nu) \big).\]
\end{lemma}

With this, we show that the curves of maximal slope for the regularized energies $\Gamma$-converge to the curve of maximal slope of the unregularized energy.

\medskip

\begin{theorem} \label{convergenceofgradflow}
Suppose the energy $\E$ satisfies  (E1)--(E2) and the mollifier $\varphi$ satisfies (M1)--(M3).
Let $\mu_\epsilon(t)\colon(0,T) \to \P_{2,R}(\Rd)$ be a curve of maximal slope of $\E_\eps$ for the strong upper gradient $g_\epsilon(\mu):=2 \|\nabla K_\epsilon * \mu\|_{L^2(\mu)}$ on the metric space $(P_{2,R}(\Rd),d_W)$. Suppose  that $\mu_\epsilon(0)$ is well-prepared, in the sense that for some $\mu(0)\in\P_{2,R}(\Rd)$,
	\[
		\mu_\eps(0)\arrow\mu(0) \text{ weak-* in $\P(\Rd)$ and } \lim_{\epsilon\to 0} \E_\epsilon(\mu_\epsilon(0)) = \E(\mu(0)).
	\]
	
Then for all $t \in (0,T)$, $\mu_\eps(t)$ has a weak-* convergent subsequence $\mu_\epsilon(t) \xrightarrow{\epsilon \to 0} \mu(t)$, and $\mu(t) \colon (0,T) \to \P_{2,R}(\Rd)$ is a curve of maximal slope of $\E$  for the strong upper gradient $g(\mu):=2 \|\nabla K * \mu\|_{L^2(\mu)}$. Furthermore, as $\eps\arrow 0$,
	\begin{gather*}
 \E_\epsilon(\mu_\epsilon(t)) \to \E(\mu(t)) \text{ for all }  \ t \in [0, T),\\
2\|\nabla K_\epsilon * \mu_\epsilon \|_{L^2(\mu_\epsilon)} \to 2\| \nabla K* \mu\|_{L^2(\mu)} \text{ in } L^2_\loc(0,T),\\ \text{and } \ |\mu'_\epsilon|_{d_W} \to |\mu'|_{d_W} \text{ in } L^2_\loc(0,T).
\end{gather*}

\end{theorem}

\begin{remark}
{\rm
For any $\mu(0)\in\P_{2,R}(\Rd)$, there exists $\mu_\eps(0)\in\P_{2,R}$ satisfying the conditions of the theorem: simply define $\mu_\eps(0)$ by convolving $\mu(0)$ with the heat kernel, as in Theorem \ref{thm:gammaconv}(ii).
}
\end{remark}

\medskip

\begin{proof}
Let $(\S_\epsilon,d_\epsilon) =(\S, d)= (\P_{2,R}(\Rd),d_W)$. Note that 
\[ \P_{2, R}(\Rd) = \{ \mu\in\P_2(\Rd) \colon \|\mu\|_{L^\infty(\Rd)} \leq R\} \]
 is closed with respect to $d_W$, thus $\P_{2,R}(\Rd)$ is a complete metric space. 
Furthermore, any $d_W$ bounded set of $\P_{2, R}(\Rd)$ is relatively compact with respect to weak-* convergence in $\P(\Rd)$, and its limit points lie in $\P_{2, R}(\Rd)$.
Given $\mu_\epsilon \in \S_\epsilon$, $\mu \in \S$ we say $\mu_\epsilon \stackrel{S}{\rightharpoonup} \mu $ if the $\mu_\eps$ converges with respect to the weak-* convergence in $\P(\Rd)$. 

We now define the strong upper gradients.
 For all $\epsilon > 0$, Remark (\ref{rem:differentiability})  ensures that
$g_\epsilon(\nu) = 2\| \nabla K_\epsilon * \nu \|_{L^2(\nu)}$ is a strong upper gradient of $\E_\epsilon$ on $\P_2(\Rd)$, hence on $\P_{2,R}(\Rd)$.
Next, we show that $g(\nu) =2 \|\nabla K * \nu\|_{L^2(\nu)}$ is a strong upper gradient of $\E$. Throughout, we use the fact that if $\mu \in L^\infty(\Rd) \cap L^1(\Rd)$, then $ K * \mu(x)$ is continuously differentiable and  $\nabla (K*\mu) = (\nabla K)* \mu$.

Suppose $\nu(t): (a,b) \to \P_{2,R}(\Rd)$ is an 2-absolutely continuous curve. The function $t \mapsto g(\nu(t))$ is measurable, since it is given by the composition of measurable functions. By Lemma \ref{stronggradientlemma},
\beqn \label{Econtinuity}
|\E(\nu(t)) - \E(\nu(s))| \leq 2 \| \nabla K*\nu(t) \|_{L^2(\nu(t))} d_W(\nu(t),\nu(s))+C_{R,d,p,q}\, \omega \big(d^2_W(\nu(t),\nu(s)) \big).
\eeqn
As in Remark \ref{rem:differentiability}, we estimate $\| \nabla K*\nu(t) \|_{L^2(\nu(t))}$ by breaking $\nabla K$ into its singular and nonsingular parts,
\begin{align*} 
&\| \nabla K*\nu(t) \|_{L^2(\nu(t))}  \leq \| \nabla K^n*\nu(t) \|_{L^2(\nu(t))} + \| \nabla K^s*\nu(t) \|_{L^2(\nu(t))} \\
&\quad \leq  \left( \int_{\Rd} \left| \int_{\Rd} \nabla K^n(x-y) \, d\nu(y,t) \right|^2 \, d\nu(x,t) \right)^{1/2}  + \sqrt{R} \| \nabla K^s*\nu(t) \|_{L^2(\Rd)} \\
&\quad \leq \int_{\Rd}  \left( \int_{\Rd} |\nabla K^n(x-y)|^2 \, d\nu(x,t) \right)^{1/2} \, d\nu(y,t)  + \sqrt{R} \| \nabla K^s \|_{L^1(\Rd)} \|\nu(t) \|_{L^2(\Rd)}.
\end{align*}
Since $\nabla K^n$ has at most linear growth, the first term is bounded by the second moments of $\nu(t)$, which are uniformly bounded for $t \in (a,b)$ since $\{\nu(t)\}_{t \in (a,b)}$ is bounded with respect to $d_W$. The second term is uniformly bounded since $\|\nu(t)\|_{L^1}= 1$ and $\|\nu(t)\|_{L^\infty} \leq R$.

By the absolute continuity of $\nu(t)$, $d_W(\nu(t),\nu(s)) \leq \int_s^t |\nu'|(r)\,dr$ for $|\nu'|(r) \in L^2(a,b)$. Combining this with (\ref{Econtinuity}) ensures that for almost every $t$
\begin{align*}
 \left| \frac{d}{dt} \E(\nu(t)) \right| &\leq \lim_{s \to t} \frac{|\E(\nu(t)) - \E(\nu(s))|}{|t-s|} \\
 &\leq \lim_{s \to t} 2 \| \nabla K*\nu(t) \|_{L^2(\nu(t))} \frac{d_W(\nu(t),\nu(s))}{|t-s|} +C_{R,d,p,q}\, \frac{\omega \big(d^2_W(\nu(t),\nu(s)) \big)}{|t-s|} \\
 &\leq 2 \| \nabla K * \nu(t) \|_{L^2(\nu(t))} |\nu'|(t) , 
 \end{align*}
 where in the last line we use that, for $\omega(x)$ defined by \eqref{omegadef}, $\omega(x^2) \leq 4x \omega(x)$ and $\lim_{x\to 0} \omega(x) =0$.
Integrating the above inequality, we conclude,
\[ |\E(\nu(t)) - \E(\nu(s))| =  \left| \int_s^t \frac{d}{dr} \E(\nu(r))\,dr \right|  \leq  \int_s^t 2 \| \nabla K * \nu(r) \|_{L^2(\nu(r))} |\nu'|(r)\,dr . \]
Thus $g(\nu) = 2\|\nabla K * \nu\|_{L^2(\nu)}$ is a strong upper gradient of $\E$.

Now we show that for all $t \in (0,T)$,  $\mu_\epsilon(t) \xrightarrow{\epsilon \to 0} \mu(t) \in \P_{2,R}(\Rd)$, with respect to weak-* convergence of probability measures. Since $\mu_\eps(t)\colon [0,\infty) \to \P_2(\Rd)$ is a curve of maximal slope for the strong upper gradient $g_\eps(\mu)=2\|\nabla K_\eps * \mu\|_{L^2(\mu)}$, applying Cauchy's inequality to (\ref{stronguppergradienteqn}) and comparing with 
(\ref{curvemaxslopeeqn}), we obtain (see \cite[Remark 1.3.3]{AGS})
\begin{align} \label{eqn:energy_identity}
\E_\eps(\mu_\eps(s)) - \E_\eps(\mu_\eps(t)) = \int_s^t |\mu_\eps'|^2(r) \, dr
\end{align}
for all $0\leq s \leq t <\infty$.
By definition of the metric local slope and the fact that $\E_\eps(\mu_\eps(t))\geq 0$,
	\beqn
			d_W^2(\mu_\eps(0),\mu_\eps(t)) \leq \left(\int_0^t |\mu_\eps^\pr|(r)\,dr\right)^2 \leq t \int_0^t |\mu_\eps^\pr|^2(t)\,dt \leq T\, \E_\eps(\mu_\eps(0))
		\nonumber
	\eeqn
for all $t \in (0,T)$. Since $\lim_{\epsilon\to 0} \E_\epsilon(\mu_\epsilon(0)) = \E(\mu(0))$, the right hand side is uniformly bounded for $\eps$ sufficiently small. Therefore, $\{\mu_\eps(t)\}_{\eps>0}$ is uniformly bounded in $\P_{2,R}(\Rd)$, and, up to a subsequence, $\mu_\eps(t)\xrightarrow{\eps \to 0}\mu(t) \in \P_{2,R}(\Rd)$ with respect to weak-*convergence in $\P(\Rd)$.

It remains to verify criteria $(i)-(iii)$ of Theorem (\ref{thm:Serfatygammagradflow}) to conclude that $\mu(t)$ is a curve of maximal slope of $\E$ for $g(\mu)$, and the corresponding energies, strong upper gradients and metric local slopes converge as $\eps\to 0$.

Criterion $(i)$ follows immediately from Theorem \ref{thm:gammaconv}, part $(i)$ and Remark \ref{rem:lsc_weak}.
To prove $(ii)$, we assume without loss of generality that there exists $0\leq C<+\infty$ so that 
\[ C = \liminf_{\epsilon \to 0} \int_0^s |\mu_\epsilon'|^2(t)\, dt. \]
Choose a subsequence $|\tilde{\mu}_\epsilon'|(t)$ so that $\lim_{\epsilon \to 0} \int_0^s |\tilde{\mu}_\epsilon'|^2(t)\, dt = C$. Then $|\tilde{\mu}_\epsilon'|(t)$ is bounded in $L^2(0,s)$ so, up to a further subsequence, it is weakly convergent to some $v(t) \in L^2(0,s)$. Consequently,  for any $0 \leq s_0 \leq s_1 \leq s$,
\[ \lim_{\epsilon \to 0} \int_{s_0}^{s_1} |\tilde{\mu}_\epsilon|(t) \, dt= \int_{s_0}^{s_1} v(t)\, dt. \]
By taking limits in the definition of the metric derivative and using the lower semicontinuity of $d_W$ with respect to weak-* convergence,
\[ d_W(\mu_\epsilon(s_0), \mu_\epsilon(s_1)) \leq \int_{s_0}^{s_1} |\mu_\epsilon'|(t)\, dt \quad \text{ yields } \quad d_W(\mu(s_0),\mu(s_1)) \leq \int_{s_0}^{s_1} v(t)\, dt. \]
By \cite[Theorem 1.1.2]{AGS}, this implies that $|\mu'|(t) \leq v(t)$ for a.e. $t \in (0,s)$. Thus, by the lower semicontinuity of the $L^2(0,s)$ norm with respect to weak convergence,
\[ \liminf_{\epsilon \to 0} \int_0^s |\mu_\epsilon'|^2(t)\, dt = \lim_{\epsilon \to 0}  \int_0^s |\tilde{\mu}_\epsilon'|^2(t)\, dt \geq \int_0^s v(t)^2\, dt \geq \int_0^s |\mu'|^2(t)\, dt. \]

Finally, we turn to $(iii)$. We assume without loss of generality that 
\[ C =   \liminf_{\epsilon \to 0} g_\epsilon(\mu_\epsilon), \]
for some $0\leq C<+\infty$.
Choose a subsequence $g_\epsilon(\mu_\epsilon)$ so that $ \lim_{\epsilon \to 0} g_\epsilon(\mu_\epsilon)  = C$. Since $\|\mu_\epsilon(t)\|_{L^\infty(\Rd)} \leq R$ for all $\epsilon >0$ and all $t \in (0,T)$, there exists a further subsequence so that $\mu_\epsilon$ converges to some limit $\nu$ in the weak-* topology of $L^\infty(\Rd)$. Since $\mu_\epsilon$ also converges to $\mu$ with respect to the weak-* topology of $\P(\Rd)$, for all $f \in C^\infty_c(\Rd)$,
\[ \int_{\Rd} f(x) \nu(x)\, dx = \lim_{\epsilon \to 0} \int_{\Rd} f(x) \mu_\epsilon(x)\, dx = \int_{\Rd} f(x) \mu(x)\, dx; \]
hence, $\nu = \mu$.

By \cite[Theorem 5.4.4]{AGS} it suffices to show 
\[ \lim_{\epsilon \to 0} \int_{\Rd} f(x) (\nabla K_\epsilon * \mu_\epsilon)(x)\, d \mu_\epsilon(x)= \int_{\Rd} f(x)( \nabla K *\mu)(x)\, d \mu(x) \quad\text{ for all } f \in C^\infty_c(\Rd). \]
Then using the convexity and the lower semicontinuity of the function $|\cdot|^2$ along with the weak-* convergence of $\mu_\eps$ yields the result.

Following a similar argument as in Lemma \ref{moveepsilontomu} and defining $\tilde{\mu}_\epsilon  = \varphi_\epsilon * \mu_\epsilon * \varphi_\epsilon$, we have 
	\beqn
	 \begin{aligned}
			\left| \int_{\Rd} f(x)( \nabla K_\epsilon *\right. &\left. \mu_\epsilon)(x)\, d \mu_\epsilon(x)- \int_{\Rd} f(x)( \nabla K *\mu)(x)\, d \mu(x) \right| \\
																					 &=\left| \int_{\Rd} f(x)( \nabla K * \tilde{\mu}_\epsilon)(x)\, d \mu_\epsilon(x) - \int_{\Rd} f(x)( \nabla K *\mu)(x)\, d \mu(x) \right| \\
																					 &\leq \left|\int_{\Rd} f(x) [(\nabla K *\tilde{\mu}_\epsilon)(x) - (\nabla K*\mu)(x)]\, d \mu_\epsilon(x) \right|\\
		 																			 &\qquad+ \left| \int_{\Rd} f(x)(\nabla K * \mu)(x)\, d \mu_\epsilon(x) - \int_{\Rd} f(x)(\nabla K * \mu)(x)\,d\mu(x) \right|\\
		 																			 &=: A_\epsilon + B_\epsilon.
		\end{aligned}
		\nonumber
	\eeqn
Since  $\mu_\eps\arrow\mu$ weak-* in $\P(\Rd)$ and $f(x)(\nabla K*\mu)(x)$ is bounded and continuous, $\lim_{\epsilon \to 0} B_\epsilon = 0$.

It remains to show $\lim_{\eps \to 0} A_\epsilon= 0$. First, note that $\tilde{\mu}_\epsilon \arrow \mu$ in the weak-* topology of $L^\infty(\Rd)$ as $\eps\arrow 0$. Indeed, for any $f \in L^1(\Rd)$, we have that
		\begin{multline}
				\left| \int_{\Rd} f(x)\, d \tilde{\mu}_\epsilon(x)   - \int_{\Rd} f(x)\, d \mu(x) \right| 																					= \left| \int_{\Rd} \varphi_\epsilon* f * \varphi_\epsilon(x)\, d \mu_\epsilon(x)  - \int_{\Rd} f(x)\, d \mu(x) \right| \\
																					 \leq R\|\varphi_\epsilon * f* \varphi_\epsilon - f \|_{L^1} + \left| \int_{\Rd} f(x)\, d \mu_\epsilon(x) - \int_{\Rd} f(x)\, d \mu(x) \right|,
			 \nonumber
		\end{multline}
where both terms approach zero as $\epsilon \to 0$. Returning to $A_\eps$,
\beqn
		\begin{aligned}
A_\epsilon &\leq \int_{\Rd} \left|f(x) (\nabla K*( \tilde{\mu}_\epsilon -\mu))(x)\right|\, \left|\mu_\epsilon(x)\right|\, dx \\
					& \leq R\, \|f\|_{L^\infty(\Rd)} \int_{\supp f} |\nabla K * ( \tilde{\mu}_\epsilon -\mu)(x)|\, dx.
		\end{aligned}
		\nonumber
\eeqn
Since the integrand has at most linear growth, it is bounded on the compact set $\supp f$, and we may apply the dominated convergence theorem, provided the integrand converges pointwise. 

When $0<q\leq 1$, $\nabla K$ is the sum of a bounded continuous function and an integrable function, and since $\tilde{\mu}_\eps \xrightarrow{\eps \to 0} \mu$ in both weak-* probability and weak-* $L^\infty(\Rd)$, the integrand converges for each $x$. On the other hand, when $1<q\leq 2$, it suffices to show that $\nabla K$ is the sum of a continuous function, which is uniformly integrable with respect to $\tilde{\mu}_\eps$, and an integrable function \cite[Lemma 5.1.7]{AGS}. In particular, is enough to show that $|x|^{q-1}$ is uniformly integrable with respect to $\tilde{\mu}_\eps$. Since $|x|\geq k$ implies that $|x|/k\geq 1$ we have that
	\[
		\lim_{k\arrow \infty} \int_{|x|\geq k}|x|^{q-1}\,d\tilde{\mu}_\eps \leq \lim_{k\arrow\infty}\frac{1}{k}\int_{|x|\geq k}|x|^q\,d\tilde{\mu}_\eps \leq \lim_{k\arrow\infty}\frac{1}{k}\E(\tilde{\mu}_\eps) = \lim_{k\arrow\infty}\frac{1}{k}\E_\eps(\mu_\eps).
	\]
Since $\mu_\eps(t)$ is a curve of maximal slope for $\E_\eps$,  \eqref{eqn:energy_identity} gives $\E_\eps(\mu_\eps(t))\leq \E_\eps(\mu_\eps(0))$ for all $t \in (0,T)$. The well-preparedness of the initial data gives $\E_\eps(\mu_\eps(0))<C$, hence the result follows.
\end{proof}

\medskip

\begin{remark}\label{rem:gradflowgeneralization}
{\rm
The extension of Theorem~\ref{convergenceofgradflow} to energies defined via kernels satisfying the more general conditions (H1)--(H5) from section~\ref{sec:generalization} requires also a generalization of the technical result of Lemma~\ref{stronggradientlemma}, which plays a crucial role in the proof of the convergence theorem. As such a generalization is beyond the scope of our current study we leave this to future studies.
}
\end{remark}

\bigskip

\noindent {\bf Acknowledgments.} This project was initiated during the Thematic Program on Variational Problems in Physics, Economics and Geometry at the Fields Institute for Research in Mathematical Sciences. The authors would like to thank the organizers and the Institute for their hospitality and support during their visits. The authors would also like to thank Andrea Bertozzi, Eric Carlen, Andres Contreras, Inwon Kim, and Xavier Lamy for useful discussions, along with the anonymous reviewers for their detailed reading and valuable comments.

\bibliographystyle{siam}
\bibliography{bibl_CrTo}

\begin{thebibliography}{10}

\bibitem{AGS}
{\sc L.~Ambrosio, N.~Gigli, and G.~Savar{\'e}}, {\em Gradient flows in metric
  spaces and in the space of probability measures}, Lectures in Mathematics ETH
  Z\"urich, Birkh\"auser Verlag, Basel, second~ed., 2008.

\bibitem{AmbrosioSerfaty}
{\sc L.~Ambrosio and S.~Serfaty}, {\em A gradient flow approach to an evolution
  problem arising in superconductivity}, Comm. Pure Appl. Math., 61 (2008),
  pp.~1495--1539.

\bibitem{AndersonGreengard}
{\sc C.~Anderson and C.~Greengard}, {\em On vortex methods}, SIAM J. Numer.
  Anal., 22 (1985), pp.~413--440.

\bibitem{BalagueCarrilloLaurentRaoul}
{\sc D.~Balagu\'e, J.~Carrillo, T.~Laurent, and G.~Raoul}, {\em Nonlocal
  interactions by repulsive-attractive potentials: Radial ins/stability}, Phys.
  D., 260 (2013), pp.~5--25.

\bibitem{BalagueCarrilloYao}
{\sc D.~Balagu{\'e}, J.A. Carrillo, and Y.~Yao}, {\em Confinement for
  repulsive-attractive kernels}, Discrete Contin. Dyn. Syst. Ser. B, 19 (2014),
  pp.~1227--1248.

\bibitem{BalagueCarrilloLaurentRaoul_Dimensionality}
{\sc D.~Balagu{\'e}, J.~A. Carrillo, T.~Laurent, and G.~Raoul}, {\em
  Dimensionality of local minimizers of the interaction energy}, Arch. Ration.
  Mech. Anal., 209 (2013), pp.~1055--1088.

\bibitem{BealeMajda1}
{\sc J.T. Beale and A.~Majda}, {\em Vortex methods. {I}. {C}onvergence in three
  dimensions}, Math. Comp., 39 (1982), pp.~1--27.

\bibitem{BealeMajda2}
\leavevmode\vrule height 2pt depth -1.6pt width 23pt, {\em Vortex methods.
  {II}. {H}igher order accuracy in two and three dimensions}, Math. Comp., 39
  (1982), pp.~29--52.

\bibitem{BenedettoCagliotiCarrilloPulvirenti}
{\sc D.~Benedetto, E.~Caglioti, J.~A. Carrillo, and M.~Pulvirenti}, {\em A
  non-{M}axwellian steady distribution for one-dimensional granular media}, J.
  Statist. Phys., 91 (1998), pp.~979--990.

\bibitem{BernoffTopaz}
{\sc A.J. Bernoff and C.M. Topaz}, {\em A primer of swarm equilibria}, SIAM J.
  Appl. Dyn. Syst., 10 (2011), pp.~212--250.

\bibitem{BertozziBrandman}
{\sc A.L. Bertozzi and J.~Brandman}, {\em Finite-time blow-up of {$L\sp
  \infty$}-weak solutions of an aggregation equation}, Commun. Math. Sci., 8
  (2010), pp.~45--65.

\bibitem{BertozziCarrilloLaurent}
{\sc A.L. Bertozzi, J.A. Carrillo, and T.~Laurent}, {\em Blow-up in
  multidimensional aggregation equations with mildly singular interaction
  kernels}, Nonlinearity, 22 (2009), pp.~683--710.

\bibitem{BertozziGarnettLaurent}
{\sc A.L. Bertozzi, J.B. Garnett, and T.~Laurent}, {\em Characterization of
  radially symmetric finite time blowup in multidimensional aggregation
  equations}, SIAM J. Math. Anal., 44 (2012), pp.~651--681.

\bibitem{BertozziLaurentLeger}
{\sc A.L. Bertozzi, T.~Laurent, and F.~L{\'e}ger}, {\em Aggregation and
  spreading via the {N}ewtonian potential: the dynamics of patch solutions},
  Math. Models Methods Appl. Sci., 22 (2012), pp.~1140005, 39.

\bibitem{BertozziLaurentRosado}
{\sc A.L. Bertozzi, T.~Laurent, and J.~Rosado}, {\em {$L\sp p$} theory for the
  multidimensional aggregation equation}, Comm. Pure Appl. Math., 64 (2011),
  pp.~45--83.

\bibitem{Bertozzietal_RingPatterns}
{\sc A.L. Bertozzi, H.~Sun, T.~Kolokolnikov, D.~Uminsky, and J.~Von~Brecht},
  {\em Ring patterns and their bifurcations in a nonlocal model of biological
  swarms}, Comm. Math. Sci., 13 (2015), pp.~955--985.
\newblock Special issue in honor of George Papanicolau's birthday.

\bibitem{BCC}
{\sc A.~Blanchet, E.A. Carlen, and J.A. Carrillo}, {\em Functional
  inequalities, thick tails and asymptotics for the critical mass
  {P}atlak-{K}eller-{S}egel model}, J. Funct. Anal., 262 (2012),
  pp.~2142--2230.

\bibitem{CCP}
{\sc J.A. Ca{\~n}izo, J.A. Carrillo, and F.S. Patacchini}, {\em Existence of
  compactly supported global minimisers for the interaction energy}, Arch.
  Ration. Mech. Anal., 217 (2015), pp.~1197--1217.

\bibitem{CaChHu}
{\sc J.A. Carrillo, M.~Chipot, and Y.~Huang}, {\em On global minimizers of
  repulsive-attractive power-law interaction energies}, Philos. Trans. R. Soc.
  Lond. Ser. A Math. Phys. Eng. Sci., 372 (2014), pp.~20130399, 13.

\bibitem{CarrilloChoiHauray}
{\sc J.A. Carrillo, Y.-P. Choi, and M.~Hauray}, {\em The derivation of swarming
  models: mean-field limit and {W}asserstein distances}, Collective Dynamics
  from Bacteria to Crowds: An Excursion Through Modeling, Analysis and
  Simulation Series{, CISM International Centre for Mechanical Sciences}, 553
  (2014), pp.~1--46.

\bibitem{CDM14}
{\sc J.A. Carrillo, M.G. Delgadino, and A.~Mellet}, {\em Regularity of local
  minimizers of the interaction energy via obstacle problems}.
\newblock Preprint, 2014.

\bibitem{5person}
{\sc J.A. Carrillo, M.~Di~Francesco, A.~Figalli, T.~Laurent, and
  D.~Slep{\v{c}}ev}, {\em Global-in-time weak measure solutions and finite-time
  aggregation for nonlocal interaction equations}, Duke Math. J., 156 (2011),
  pp.~229--271.

\bibitem{CarrilloLisiniMainini}
{\sc J.A. Carrillo, S.~Lisini, and E.~Mainini}, {\em Uniqueness for
  {K}eller-{S}egel-type chemotaxis models}, Discrete Contin. Dyn. Syst., 34
  (2014), pp.~1319--1338.

\bibitem{CarrilloMcCannVillani}
{\sc J.A. Carrillo, R.J. McCann, and C.~Villani}, {\em Kinetic equilibration
  rates for granular media and related equations: entropy dissipation and mass
  transportation estimates}, Rev. Mat. Iberoamericana, 19 (2003),
  pp.~971--1018.

\bibitem{CarrilloMcCannVillani2}
\leavevmode\vrule height 2pt depth -1.6pt width 23pt, {\em Contractions in the
  2-{W}asserstein length space and thermalization of granular media}, Arch.
  Ration. Mech. Anal., 179 (2006), pp.~217--263.

\bibitem{CarrilloRosado}
{\sc J.A. Carrillo and J.~Rosado}, {\em Uniqueness of bounded solutions to
  aggregation equations by optimal transport methods}, in European {C}ongress
  of {M}athematics, Eur. Math. Soc., Z\"urich, 2010, pp.~3--16.

\bibitem{ChafaiGozlanZitt}
{\sc D.~Chafa{\"{\i}}, N.~Gozlan, and P.-A. Zitt}, {\em First-order global
  asymptotics for confined particles with singular pair repulsion}, Ann. Appl.
  Probab., 24 (2014), pp.~2371--2413.

\bibitem{ChFeTo14}
{\sc R.~Choksi, R.C. Fetecau, and I.~Topaloglu}, {\em On minimizers of
  interaction functionals with competing attractive and repulsive potentials},
  Ann. Inst. H. Poincar\'e Anal. Non Lin\'eaire,  (2015).
\newblock to appear.

\bibitem{ChuangHuangDorsognaBertozzi}
{\sc Y.-L. Chuang, Y.R. Huang, M.R. D'Orsogna, and A.L. Bertozzi}, {\em
  Multi-vehicle flocking: scalability of cooperative control algorithms using
  pairwise potentials}, IEEE International Conference on Robotics and
  Automation,  (2007), pp.~2292--2299.

\bibitem{CraigOsgood}
{\sc K~Craig}, {\em Conditions for well-posedness of non-convex gradient flow
  in the {W}asserstien metric}.
\newblock in progress, 2015.

\bibitem{CrBe14}
{\sc K.~Craig and A.~Bertozzi}, {\em A blob method for the aggregation
  equation}, Math. Comp.,  (2015).
\newblock to appear.

\bibitem{Dong}
{\sc H.~Dong}, {\em The aggregation equation with power-law kernels:
  ill-posedness, mass concentration and similarity solutions}, Comm. Math.
  Phys., 304 (2011), pp.~649--664.

\bibitem{DoyeWalesBerry}
{\sc J.P.K. Doye, D.J. Wales, and R.S. Berry}, {\em The effect of the range of
  the potential on the structures of clusters}, J. Chem. Phys., 103 (1995),
  pp.~4234--4249.

\bibitem{FellnerRaoul2}
{\sc K.~Fellner and G.~Raoul}, {\em Stable stationary states of non-local
  interaction equations}, Math. Models Methods Appl. Sci., 20 (2010),
  pp.~2267--2291.

\bibitem{FetecauHuang}
{\sc R.C. Fetecau and Y.~Huang}, {\em Equilibria of biological aggregations
  with nonlocal repulsive-attractive interactions}, Phys. D, 260 (2013),
  pp.~49--64.

\bibitem{FetecauHuangKolokolnikov}
{\sc R.C. Fetecau, Y.~Huang, and T.~Kolokolnikov}, {\em Swarm dynamics and
  equilibria for a nonlocal aggregation model}, Nonlinearity, 24 (2011),
  pp.~2681--2716.

\bibitem{GiSh84}
{\sc C.R. Givens and R.M. Shortt}, {\em A class of {W}asserstein metrics for
  probability distributions}, Michigan Math. J., 31 (1984), pp.~231--240.

\bibitem{HaganChandler}
{\sc M.F. Hagan and D.~Chandler}, {\em Dynamic pathways for viral capsid
  assembly}, Biophysical Journal, 91 (2006), pp.~42--54.

\bibitem{HuangBertozzi2}
{\sc Y.~Huang and A.L. Bertozzi}, {\em Self-similar blowup solutions to an
  aggregation equation in {$\bold R\sp n$}}, SIAM J. Appl. Math., 70 (2010),
  pp.~2582--2603.

\bibitem{Kolokolnikovetal_StabilityRingPatterns}
{\sc T.~Kolokolnikov, H.~Sun, D.~Uminsky, and A.L. Bertozzi}, {\em Stability of
  ring patterns arising from two-dimensional particle interactions}, Phys. Rev.
  E, 84 (2011).

\bibitem{LiToscani}
{\sc H.~Li and G.~Toscani}, {\em Long-time asymptotics of kinetic models of
  granular flows}, Arch. Ration. Mech. Anal., 172 (2004), pp.~407--428.

\bibitem{LiLo}
{\sc E.H. Lieb and M.~Loss}, {\em Analysis}, vol.~14 of Graduate Studies in
  Mathematics, American Mathematical Society, Providence, RI, second~ed., 2001.

\bibitem{Lions84}
{\sc P.-L. Lions}, {\em The concentration-compactness principle in the calculus
  of variations. {T}he locally compact case. {I}}, Ann. Inst. H. Poincar\'e
  Anal. Non Lin\'eaire, 1 (1984), pp.~109--145.

\bibitem{Mainini}
{\sc E.~Mainini}, {\em A global uniqueness result for an evolution problem
  arising in superconductivity}, Boll. Unione Mat. Ital. (9), 2 (2009),
  pp.~509--528.

\bibitem{MaininiGinzburg}
\leavevmode\vrule height 2pt depth -1.6pt width 23pt, {\em Well-posedness for a
  mean field model of {G}inzburg-{L}andau vortices with opposite degrees},
  NoDEA Nonlinear Differential Equations Appl., 19 (2012), pp.~133--158.

\bibitem{McCann2}
{\sc R.J. McCann}, {\em Stable rotating binary stars and fluid in a tube},
  Houston J. Math., 32 (2006), pp.~603--631.

\bibitem{MogilnerEdelstein}
{\sc A.~Mogilner and L.~Edelstein-Keshet}, {\em A non-local model for a swarm},
  J. Math. Biol., 38 (1999), pp.~534--570.

\bibitem{MogilnerEdelsteinBent}
{\sc A.~Mogilner, L.~Edelstein-Keshet, L.~Bent, and A.~Spiros}, {\em Mutual
  interactions, potentials, and individual distance in a social aggregation},
  J. Math. Biol., 47 (2003), pp.~353--389.

\bibitem{PereaGomezElosegui}
{\sc L.~Perea, G.~G\'omez, and P.~Elosegui}, {\em Extension of the
  {C}ucker--{S}male control law to space flight formations}, AIAA J. of
  Guidance, Control, and Dynamics, 32 (2009), pp.~527--537.

\bibitem{PetzHiai}
{\sc D.~Petz and F.~Hiai}, {\em Logarithmic energy as an entropy functional},
  in Advances in differential equations and mathematical physics ({A}tlanta,
  {GA}, 1997), vol.~217 of Contemp. Math., Amer. Math. Soc., Providence, RI,
  1998, pp.~205--221.

\bibitem{Poupaud}
{\sc F.~Poupaud}, {\em Diagonal defect measures, adhesion dynamics and {E}uler
  equation}, Methods Appl. Anal., 9 (2002), pp.~533--561.

\bibitem{RechtsmanStillingerTorquato}
{\sc M.C. Rechtsman, F.H. Stillinger, and S.~Torquato}, {\em Optimized
  interactions for targeted self- assembly: application to a honeycomb
  lattice}, Phys. Rev. Lett., 95 (2005).

\bibitem{Serfaty}
{\sc S.~Serfaty}, {\em Gamma-convergence of gradient flows on {H}ilbert and
  metric spaces and applications}, Discrete Contin. Dyn. Syst., 31 (2011),
  pp.~1427--1451.

\bibitem{SiSlTo2014}
{\sc R.~Simione, D.~Slep{\v{c}}ev, and I.~Topaloglu}, {\em Existence of ground
  states of nonlocal-interaction energies}, J. Stat. Phys., 159 (2015),
  pp.~972--986.

\bibitem{SteinSingularIntegrals}
{\sc E.M. Stein}, {\em Singular integrals and differentiability properties of
  functions}, Princeton Mathematical Series, No. 30, Princeton University
  Press, Princeton, N.J., 1970.

\bibitem{Struwe}
{\sc M.~Struwe}, {\em Variational Methods: Applications to Nonlinear Partial
  Differential Equations and Hamiltonian Systems}, Springer-Verlag, Berlin,
  third~ed., 2000.

\bibitem{SunUminskyBertozzi}
{\sc H.~Sun, D.~Uminsky, and A.L. Bertozzi}, {\em Stability and clustering of
  self-similar solutions of aggregation equations}, J. Math. Phys., 53 (2012),
  pp.~115610, 18.

\bibitem{TopazBertozzi1}
{\sc C.M. Topaz and A.L. Bertozzi}, {\em Swarming patterns in a two-dimensional
  kinematic model for biological groups}, SIAM J. Appl. Math., 65 (2004),
  pp.~152--174.

\bibitem{TopazBertozzi2}
{\sc C.M. Topaz, A.L. Bertozzi, and M.A. Lewis}, {\em A nonlocal continuum
  model for biological aggregation}, Bull. Math. Biol., 68 (2006),
  pp.~1601--1623.

\bibitem{van1996weak}
{\sc A.W. van~der Vaart and J.~Wellner}, {\em Weak Convergence and Empirical
  Processes: With Applications to Statistics}, Springer Series in Statistics,
  Springer, 1996.

\bibitem{Villani}
{\sc C.~Villani}, {\em Topics in optimal transportation}, vol.~58 of Graduate
  Studies in Mathematics, American Mathematical Society, Providence, RI, 2003.

\bibitem{Wales}
{\sc D.J. Wales}, {\em Energy landscapes of clusters bound by short-ranged
  potentials}, Chem. Eur. J. Chem. Phys., 11 (2010), pp.~2491--2494.

\bibitem{YaoBertozzi}
{\sc Y.~Yao and A.L. Bertozzi}, {\em Blow-up dynamics for the aggregation
  equation with degenerate diffusion}, Phys. D, 260 (2013), pp.~77--89.

\end{thebibliography}

\end{document}